\documentclass[review,1p]{elsarticle}
\usepackage{graphicx,amssymb,mathrsfs,amsmath,amsfonts,dsfont}
\usepackage{hyperref}
\usepackage{lineno}
\usepackage{subfigure}
\usepackage{algorithm}
\usepackage{algorithmic}
\usepackage{amsbsy}
\usepackage[mathscr]{euscript}
\usepackage{framed}

\newcommand{\X}{\boldsymbol{\mathscr{X}}}

\makeatletter

\newcommand{\Rmnum}[1]{\expandafter\@slowromancap\romannumeral #1@}
\makeatother
\newtheorem{theorem}{Theorem}[section]
\newtheorem{lemma}{Lemma}[section]

\newtheorem{remark}{Remark}[section]

\newproof{proof}{Proof}

\usepackage{tablefootnote}
\usepackage{url}

\usepackage{amsfonts}
\usepackage{mathrsfs}
\usepackage{bbm}

\usepackage{color}
\usepackage{graphicx}
\usepackage{epstopdf}
\usepackage{subfigure}
\usepackage{epsfig}
\usepackage{amssymb}
\usepackage{epsf}
\usepackage{subfigure}
\usepackage{multirow}
\usepackage{booktabs}
\usepackage{amsmath}
\usepackage{diagbox}
\journal{}

\begin{document}

\begin{frontmatter}

\title{Efficient Dual ADMMs for Sparse Compressive Sensing MRI Reconstruction}



\author[a1]{Yanyun Ding}
\ead{dingyanyunhenu@163.com}
\author[a2]{Peili Li}
\ead{lipeili@whu.edu.cn}
\author[a3]{Yunhai Xiao\footnote{Corresponding author}}
\ead{yhxiao@henu.edu.cn}
\author[a1]{Haibin Zhang}
\ead{zhanghaibin@bjut.edu.cn}

\address[a1]{College of Applied Science, Beijing University of Technology,
Beijing 100124, China}
\address[a2]{School of Mathematics and Statistics, Wuhan University, Wuhan 430072, China}
\address[a3]{School of Mathematics and Statistics,
Henan University, Kaifeng 475000, China}

\begin{abstract}
Magnetic Resonance Imaging (MRI) is a kind of medical imaging technology used for diagnostic imaging of
diseases, but its image quality may be suffered by the long acquisition time.
The compressive sensing (CS) based
strategy may decrease the reconstruction time greatly, but it needs efficient reconstruction algorithms
to produce high-quality and reliable images.
This paper focuses on
the algorithmic improvement for the sparse reconstruction of CS-MRI, especially considering
a non-smooth convex minimization problem which is composed of the sum of a total variation regularization term and a $\ell_1$-norm term of the wavelet transformation.
The partly motivation of targeting the dual problem is that the dual variables are involved in relatively low-dimensional subspace.
Instead of solving the primal model as usual, we turn our attention to its associated dual model composed of
three variable blocks and two separable non-smooth function blocks.
However, the directly extended alternating direction method of multipliers (ADMM)
must be avoided because it may be divergent, although it usually performs well numerically.
In order to solve the problem, we employ a symmetric Gauss-Seidel (sGS) technique based ADMM.
Compared with the directly extended ADMM, this method only needs one additional iteration, but its convergence can be guaranteed theoretically.
Besides,  we also propose a generalized variant of ADMM because this method has been illustrated to be efficient for solving
semidefinite programming in the past few years.
Finally, we do extensive experiments on MRI reconstruction using some simulated and real MRI images
under different sampling patterns and ratios.
The numerical results demonstrate that the proposed algorithms significantly
achieve high reconstruction accuracies with fast computational speed.
\end{abstract}

\begin{keyword}
Magnetic resonance imaging; non-smooth convex minimization; compressive sensing; symmetric Gauss-Seidel iteration;  alternating direction method of multipliers.

\end{keyword}

\end{frontmatter}

\section{Introduction}
%
%
%
%
Magnetic Resonance Imaging (MRI) has become an indispensable imaging tool for diagnosing
and evaluating a host of conditions and diseases.
The MRI data acquisition system can be characterized as follows:
\begin{equation}\label{orig}
b=\mathcal{F} u+\epsilon,
\end{equation}
where $u$ denotes the vectorized MR image formed by row/column concatenation of an image matrix, $b$ is the collected $k$-space data, $\mathcal{F}$ is the Fourier operator which maps
the image space to the $k$-space,  and $\epsilon$ is the system noise assumed to be
normally distributed. The goal is to reconstruct $u$ from the given collected $k$-space data $b$.  However, this data acquisition process
is quite time consuming due to physiological and hardware constraints.

The compressive sensing (CS) \cite{crtsix,ddseven} MRI is an effective approach allowing for
data sampling much lower without significantly degrading the image quality. Let $\mathcal{P}$ be an under-sampling mask on the $k$-space as shown in \cite{gbkthirteen,jsktwelve,psleleven}, and further let $K:=\mathcal{P}\mathcal{F}$, i.e., $K:\mathbb{R}^d\rightarrow\mathbb{R}^p$ $(p\ll d)$ is a partial Fourier transform matrix \cite{chan30}. The undersampling process form
of the sparse CS-MRI reconstruction can be mathematically modeled as
\begin{equation*}
  \min_{u\in\mathbb{R}^d} \ \{F(u): \  K u=b\},
\end{equation*}
where $F: \mathbb{R}^{d}\rightarrow\mathbb{R}$ is a sparse-promoting function, and $b\in \mathbb{R}^{p}$ represents the undersampled data.
We should mention that $F$ is a good sparse approximation under a certain transform, such as gradient operator \cite{tcse23} or wavelet transform\cite{cai26}.
The earlier works of Lustig et al. \cite{ldfive} and He et al. \cite{ltsixteen} modeled the MRI reconstruction as a linear combination of
wavelet sparsity and total variation (TV) regularization \cite{solr24}.
More precisely, denote $W\in \mathbb{R}^{q\times d}$ be a Haar wavelet transform matrix and
$\Lambda:=\mbox{diag}(\lambda_{1},\lambda_{2},...,\lambda_{q})$ with $\lambda_{i}\geq 0$
be a diagonal matrix, the  CS-MRI model using TV and wavelet is modeled as follows
\begin{equation}\label{mripro}
  \min_{u\in\mathbb{R}^d} \ \{\mu\|u\|_{\text{TV}}+\|\Lambda Wu\|_{1} \ : \ K u=b\},
\end{equation}
where $\|\cdot\|_{\text{TV}}$ is a discretization of
TV regularization, and $\mu>0$ is a regularization parameter.

It is observed that both terms involved in the objective function of (\ref{mripro}) are not differentiable, which
causes main challenges for solving. Over the past decades, great efforts have been made to tackle this difficulty from different aspects.
For examples, the nonlinear inverse scale space method of He et al. \cite{ltsixteen} has the capability to solve (\ref{mripro}) approximately
and has been demonstrated to be more straightforward and more efficient.
The operator splitting method of Ma et al. \cite{MAIEEE}
targets to an inclusion problem resulting from the first-order optimality condition of a related problem of (\ref{mripro}).
It is worthy of noting that both algorithms specifically solve a penalized variant of problem (\ref{mripro}) rather than
itself. The recent work of Li et al. \cite{lxzone} solves problem (\ref{mripro}) directly based on the fact
that the dual formulation of (\ref{mripro}) is essentially a convex composite optimization problem with separable structures. Consequently, Li et al. \cite{lxzone}
developed a two-step fixed-point proximity algorithm (2SFPPA) and demonstrated its numerical efficiency through a series of experiments.
Besides, there are other reconstruction approaches that can yield high quality images without
significantly increasing the computational cost, e.g., the methods in \cite{jia1,jia2} not only keep regularity of smooth part of images,
but also preserve the edges in images; the method in \cite{x50} makes use of similarity in images to establish a
general patch-based nonlocal operator to provide sparse representation of similar
patches; the method in \cite{jia4} uses the dependencies of
the wavelet domain coefficients to accelerate
MRI data acquisition.

Unlike almost all the methods mentioned above, in this paper, we focus on the dual problem of (\ref{mripro}) which has
favourable structures of three separable variables blocks and two separable non-smooth  blocks.
Nevertheless, we note that the directly
extended alternating direction method of multipliers (ADMM) must be avoided because the type of this method might diverge if the
non-smooth blocks exceed two \cite{cai35}. To resolve this dilemma, we employ the simple but very powerful symmetric Gauss-Seidel (sGS) technique
\cite{li39,lskthree} which updates one of the variables merely twice, but it can guarantee convergence theoretically.
The advantage of using the sGS technique
is that it decomposes a large problem into several smaller pieces and then solves them correspondingly via their favorable structures, see e.g., \cite{CHENL,DINGXIAO,LIXIAO,WANGXIAO,xiao40}. As a product, we
also apply the generalized ADMM of Xiao et al. \cite{xiao40} for further performance evaluations.
We do numerical experiments on MRI reconstruction under different sampling patterns and ratios to demonstrate that the proposed algorithms significantly
achieve medium reconstruction accuracies with fast computational speed.
We have to stress, though, that the classic ADMM with unit step-length
is actually the Douglas-Rachford splitting method to the sum of two maximal monotone operators resulting from
the dual formulation, our another motivation of targeting to the dual problem is that the linear operator $KK^\top:\mathbb{R}^p\rightarrow\mathbb{R}^p$ is involved,
instead of $K^\top K:\mathbb{R}^d\rightarrow\mathbb{R}^d$ when ADMMs are used to the primal problem (\ref{mripro}) directly.

The remaining parts of this paper are organized as follows. In section \ref{preli}, we summarize some
basic definitions or concepts in convex analysis and imaging science, and review the recent
developments of ADMMs for subsequent algorithm's construction.  In section \ref{prepsec}, we do some preparations for solving the model (\ref{mripro}).
In section \ref{algsec}, we state
our motivation, and present a couple of sGS technique based ADMMs subsequently.  In section \ref{numsec},
we test both presented algorithms and do performance comparisons
by several numerical experiments.  Finally, we conclude our paper with some remarks in section \ref{finsec}.

\section{Preliminary results}\label{preli}
In this section, we list some basic concepts in convex analysis, review a couple of semi-proximal ADMMs for separable convex
optimization.
We quickly recall some preliminary results in convex analysis of Rockafellar \cite{rock41}. Let $\mathcal{H}$ be a finite-dimensional real Euclidean space endowed
with an inner product $\langle \cdot, \cdot \rangle$ and its induced norm $\|\cdot\|$ respectively.
A subset $\mathcal{C} \subseteq \mathcal{H}$ is said to be convex if $(1-\lambda)x+\lambda y\in \mathcal{C}$
whenever $x,y\in \mathcal{C}$, and $\lambda\in (0,1)$.
The relative interior of $\mathcal{C}$ is denoted by ri$(\mathcal{C})$. If $\mathcal{C}$ is a closed convex set in $\mathcal{H}$, for any $z\in \mathcal{H}$,
let $\Pi_{\mathcal{C}}(z)$ denote the metric projection of $z$
onto $\mathcal{C}$, which is the optimal solution of the minimization problem $\min_{y}\{\| y-z \|^2 \mid y\in \mathcal{C}\}$. Given $x\in\X$, the orthogonal projection onto the $\ell_{\infty}$-norm ball $\Pi_{\mathcal{B}^{(r)}_{\infty}}(x)$ with radius $r>0$ is expressed as
\begin{align}\label{orp}
\Pi_{\mathcal{B}^{(r)}_{\infty}}(x)=\min\{r,\max\{x,-r\}\}.
\end{align}
Similarly, for $\ell_2$-norm, the projection can be given explicitly by
$$
\Pi_{\mathcal{B}^{(r)}_{2}}(x)=\frac{x}{\|x\|_2}\min\{r,\|x\|_2\}.
$$
Moreover,
a subset $\mathcal{K}\subseteq \mathcal{H}$
is called a cone if it is closed under positive scalar multiplication, i.e., $\lambda x\in \mathcal{K}$ with $\lambda>0$ if $x\in\mathcal{K}$.
A cone $\mathcal{K}$ is called a convex cone if it is convex.

Let $f:\mathcal{H}\rightarrow (-\infty, +\infty]$ be a closed proper convex function, we denote by dom$(f)$ the effective domain of $f$, namely,
$\text{dom}(f)=\{x| f(x)< +\infty\}$.
The subdifferential of $f$ at point $x\in \mathcal{H}$ is the set defined by
$$
 \partial f(x):=\{ y\in \mathcal{H} \ | \  f(z)\geq f(x)+\langle y,z-x\rangle, \ \forall \  z\in \mathcal{H}\}{\color{blue}.}
$$
Obviously, $\partial f(x)$ is a closed convex set while it
is not empty. The multi-valued operator $\partial f:x \rightrightarrows \partial f(x)$
is shown to be maximal monotone.
Let $f^{*}$ denote the convex conjugate of $f$, i.e.,
\begin{align*}
 f^{*}(y): =\sup_x\{\langle y, x \rangle-f(x) \ | \ x\in \mathcal{H} \}=-\inf_x\{f(x)-\langle y, x \rangle \ | \ x\in \mathcal{H}\}.
\end{align*}
For a nonempty closed convex set $\mathcal{C}$, the symbol $\delta_{\mathcal{C}}(x)$ represents the indicator function over $\mathcal{C}$ such that $\delta_{\mathcal{C}}(x)=0$ if $x\in\mathcal{C}$ and $+\infty$ otherwise.
The conjugate of an indicator function $\delta_{\mathcal{C}}(x)$ is named support function  defined by
$\delta^{*}_{\mathcal{C}}(x)=\sup \{ \langle x, y\rangle | y\in \mathcal{C}\}$, and the subdifferential of
$\delta_{\mathcal{C}}(x)$ at $x$ is the normal cone of $\mathcal{C}$, i.e, $\partial\delta_{\mathcal{C}}(x) =\mathcal{N}_{\mathcal{C}}(x)$.
It is not hard to deduce that the Fenchel conjugate of $\|x\|_{p}$ is $\|x\|_{p}^*=\delta_{\mathcal{B}^{(1)}_{q}}(x)$ where $\mathcal{B}^{(1)}_{q}:=\{x | \|x\|_q\leq 1\}$ and $1/p+1/q=1$.

The Moreau-Yosida regularization \cite{rock43} of a closed proper convex function $f$ at $x\in \mathcal{H}$ is defined as
\begin{equation}\label{MY}
  \psi_{f}(x):=\min_{y\in \mathcal{H}} \{f(y)+\frac{1}{2}\| y-x\|^{2}\}.
\end{equation}
For any $x\in \mathcal{H}$, problem (\ref{MY}) has an unique optimal solution, which is called the proximal point of $x$ associated
with $f$, i.e.,
\begin{equation*}
  \mbox{prox}_{f}(x):=\mbox{arg} \min_{y\in \mathcal{H}} \{f(y)+\frac{1}{2}\| y-x\|^{2}\}.
\end{equation*}
In particular, the proximal point of $x$ associated with an indicator function $\delta_{\mathcal{C}}(x)$ reduces to the metric projection of $x$ onto $\mathcal{C}$, i.e.,
\begin{equation}\label{proj}
\text{prox}_{\delta_{\mathcal{C}}}(x)=\text{arg}\min_{y\in \mathcal{H}} \{\delta_{\mathcal{C}}(y)+\frac{1}{2}\| y-x \|^{2} \} = \text{arg}\min_{y\in \mathcal{C}} \{\frac{1}{2}\| y-x \|^{2}\}=\Pi_{\mathcal{C}}(x).
\end{equation}

We now turn to briefly review the definition of the discrete form of TV regularization \cite{tcse23,solr24}.
To simplify, we consider a $2$-dimensional grayscale image $U$ with size $d_{1}\times d_{2}$. The isotropic
TV is defined by
\begin{equation}\label{tvdef}
\|U\|_{\text{TV}}=\sum_{i=1}^{d_1}\sum_{j=1}^{d_2}\|\nabla U\|,
\end{equation}
where $\nabla$ denotes the forward finite difference operator on the vertical and horizonal directions, i.e.,
$$
  (\nabla U)_{i,j}=((\nabla U)^{1}_{i,j},(\nabla U)^{2}_{i,j})
$$
with
$$
  (\nabla U)^{1}_{i,j}=\left\{
\begin{array}{ll}
  U_{i+1,j}-U_{i,j},\ & \mbox{if} \ i<d_{1}, \\
  U_{i,1}-U_{i,j},\ & \mbox{if}\  i=d_{1},
\end{array}
\right.
 \quad \text{and} \quad
  (\nabla U)^{2}_{i,j}=\left\{
\begin{array}{ll}
  U_{i,j+1}-U_{i,j},\ & \mbox{if} \ j<d_{2}, \\
  U_{1,j}-U_{i,j},\ & \mbox{if}\ j =d_{2},
\end{array}
\right.
$$
for $i=1,...,d_{1}$ and $j=1,...,d_{2}$. We note that the $\ell_2$-norm in (\ref{tvdef}) can be replaced by the
$\ell_1$-norm, in which case the resulting TV is an anisotropic discretization.
We point out that,  the isotropic TV is often preferred over any anisotropic ones, both
types of discretizations lead to the so-called staircasing effects.

\section{Dual formulation and optimality condition}\label{prepsec}

In this section, we do some necessary preparations for subsequent algorithms' development.
At the first place, we reformulate the TV regularization $\|\cdot\|_{\text{TV}}$ as
a function composed with a linear mapping. We note that all the notations used here are the same as
those in \cite{lxzone} for convenience. The matrix Kronecker product is denoted as $\otimes$, and a matrix $B$ with
size $2d\times d$ is defined as
$$
  B=\left[
  \begin{array}{c}
  I_{d_{2}}\otimes D_{d_{1}}\\
  D_{d_{2}}\otimes I_{d_{1}}\\
  \end{array}
  \right],
$$
where $I$ is an identity matrix with appropriate size, and $D_{r}$ is a $r\times r$ difference matrix with the following form
\begin{equation*}
  D_{r}=\left[
   \begin{array}{cccc}
   1  &  &  & -1\\
   -1 &1 &  &   \\
      &\ddots&\ddots&  \\
     & & -1&1\\
   \end{array}
   \right].
\end{equation*}
Moreover, for any $y\in\mathbb{R}^{2d}$, we define a function $\psi:\mathbb{R}^{2d}\rightarrow \mathbb{R}$ as
$$
  \psi(y)=\sum^{d}_{i=1}\| [y_{i},y_{d+i}]^{\top}\|.
$$
With these definitions, then the isotropic TV in (\ref{mripro}) can be expressed as
\begin{equation}\label{tvtospi}
\|u\|_{\text{TV}} =\psi(Bu).
\end{equation}
Moreover, for any $y\in\mathbb{R}^q$, we define a convex function $\varphi:\mathbb{R}^{q}\rightarrow\mathbb{R}$ as
$
\| \Lambda y\|_{1}=\varphi(y),
$
which yields for $W\in\mathbb{R}^{q\times d}$ such that $Wu\in\mathbb{R}^q$ and
\begin{equation}\label{tvtospi2}
\varphi(Wu)=\|\Lambda Wu\|_1.
\end{equation}
Substituting (\ref{tvtospi}) and (\ref{tvtospi2}) into (\ref{mripro}) and recalling the definition of indicator function, it
shows that the problem (\ref{mripro}) can be reformulated as
\begin{equation}\label{mriref}
  \min_{u\in\mathbb{R}^d} \ \big\{\mu\psi(Bu)+\varphi(Wu)+\delta_{\{b\}}(K u)\big \},
\end{equation}
where $\delta_{\{b\}}(K u)$ implies $\delta_{ \{b\} }(K u)=0$ if $K u =b$ and $+\infty$ otherwise. Furthermore,
denote $Bu=z_{1}\in\mathbb{R}^{2d}$, $Wu=z_{2}\in\mathbb{R}^q$ and $Ku=z_{3}\in\mathbb{R}^p$, then problem (\ref{mriref}) can
be rewritten equivalently as
\begin{equation}\label{mrireftwo}
\begin{array}{rl}
 \min\limits_{u,z_1,z_2,z_3} & \mu\psi(z_{1})+\varphi(z_{2})+\delta_{\{b\}}(z_{3})\\[2mm]
 \mbox{s.t.}
&Bu=z_{1}, \ Wu=z_{2}, \ K u=z_{3}.
\end{array}
\end{equation}

The Lagrangian function associated with problem (\ref{mrireftwo}) is given by
\begin{align*}
\mathcal{L}(u,z_{1},z_{2},z_{3};x_{1},x_{2},x_{3})&= \mu\psi(z_{1})+\varphi(z_{2})+\delta_{\{b\}}(z_{3})\\
&+\langle Bu-z_{1}, x_{1}\rangle+\langle Wu-z_{2}, x_{2}\rangle+\langle Ku-z_{3}, x_{3}\rangle,
\end{align*}
where $x_{1}\in\mathbb{R}^{2d}$, $x_{2}\in\mathbb{R}^q$ and $x_{3}\in\mathbb{R}^p$ are multipliers associated with constrains. The
Lagrangian dual function of  (\ref{mrireftwo}) is to minimize $\mathcal{L}(u,z_{1},z_{2},z_{3};x_{1},x_{2},x_{3})$ over
$(u,z_1,z_2,z_3)$, that is
\begin{align*}
   D(x_1,x_2,x_3)
  =& \inf_{u,z_{1},z_{2},z_{3}} \mathcal{L}(u,z_{1},z_{2},z_{3};x_{1},x_{2},x_{3})\\[1mm]
  =&-(\mu\psi)^{*}(x_{1})-\varphi^{*}(x_{2})-\delta^{*}_{\{b\}}(x_{3}),
\end{align*}
by noting the definition of the conjugate function and the fact that $B^{\top}x_{1}+W^{\top}x_{2}+K^{\top}x_{3}=0$. The
Lagrangian dual problem of the original (\ref{mrireftwo}) is to maximize this dual function $D(x_1,x_2,x_3)$,
which can equivalently be written as the following minimization problem with three separate blocks of variables and a single linear equality constraint:
\begin{equation}\label{dualmriref}
\begin{array}{rl}
  \min\limits_{x_1,x_2,x_3} \ & \ (\mu\psi)^{*}(x_{1})+\varphi^{*}(x_{2})+\delta^{*}_{\{b\}}(x_{3})\\[2mm]
\text{s.t.} & B^{\top}x_{1}+W^{\top}x_{2}+K^{\top}x_{3}=0\\[2mm]
 & x_{1}\in\mathbb{R}^{2d}, \ x_{2}\in\mathbb{R}^{q}, \ x_{3}\in\mathbb{R}^{p}.
\end{array}
\end{equation}
From the properties of conjugate for norm functions reviewed previously, we know that
$
(\mu\psi)^{*}(x_{1})=\delta_{\mathcal{B}_2^{(\mu)}}(x_1),
$
where $\mathcal{B}_2^{(\mu)}=\{y\in\mathbb{R}^{2d} | \|[y_i,y_{d+i}]^\top\|\leq \mu, \ i=1,\ldots,d  \}$ and
$
\varphi^{*}(x_{2})=\delta_{\mathcal{B}_{\infty}^{(\lambda)}}(x_2),
$
where $\mathcal{B}_{\infty}^{(\lambda)}=\{y\in\mathbb{R}^q | |y_i|\leq\lambda_i, \ i=1,\ldots,q \}$. Then problem (\ref{dualmriref})
transforms into the following form
\begin{equation}\label{dual2}
\begin{array}{rl}
  \min\limits_{x_1,x_2,x_3} \ & \ \delta_{\mathcal{B}_2^{(\mu)}}(x_1)+\delta_{\mathcal{B}_{\infty}^{(\lambda)}}(x_2)+\langle b,x_3\rangle\\[2mm]
\text{s.t.} & B^{\top}x_{1}+W^{\top}x_{2}+K^{\top}x_{3}=0\\[2mm]
 & x_{1}\in\mathbb{R}^{2d}, \ x_{2}\in\mathbb{R}^{q}, \ x_{3}\in\mathbb{R}^{p}.
\end{array}
\end{equation}
The Lagrangian function associated to the dual problem (\ref{dual2}) takes the following form
$$\mathcal{L}(x_1,x_2,x_3;u)=\delta_{\mathcal{B}_2^{(\mu)}}(x_1)+\delta_{\mathcal{B}_{\infty}^{(\lambda)}}(x_2)+\langle b,x_3\rangle-\langle u,B^{\top}x_{1}+W^{\top}x_{2}+K^{\top}x_{3}\rangle,$$
where $u\in\mathbb{R}^{d}$ is a multiplier.
Suppose that $(\bar{x}_1,\bar{x}_2,\bar{x}_3)$ is the optimal solution of problem (\ref{dual2}). Then there exists a Lagrangian multiplier $\bar{u}$ such that the following KKT system holds,
\begin{equation}\label{KKTnew}
\left\{
\begin{array}{llll}
0 \in \mathcal{N}_{\mathcal{B}_2^{(\mu)}}(\bar{x}_{1})-B \bar{u},\\
0 \in \mathcal{N}_{\mathcal{B}_{\infty}^{(\lambda)}}(\bar{x}_{2})-W \bar{u},\\
0=b-K \bar{u},\\
0=B^{\top}\bar{x}_{1}+W^{\top}\bar{x}_{2}+K^{\top}\bar{x}_{3},
\end{array}
\right.
\end{equation}
where $\mathcal{N}_{\mathcal{B}_2^{(\mu)}}(\bar{x}_{1})$ (resp. $\mathcal{N}_{\mathcal{B}_{\infty}^{(\lambda)}}(\bar{x}_{2})$) is the normal cone to $\mathcal{B}_2^{(\mu)}$ (resp. $\mathcal{B}_{\infty}^{(\lambda)}$) at $\bar x_{1} \in \mathcal{B}_2^{(\mu)}$ (resp.
$\bar x_{2} \in \mathcal{B}_{\infty}^{(\lambda)} $).

The model (\ref{dual2})  has separable structure in terms of both the objective function and the constraint, and thus, it falls into the
framework of ADMM. The directly extended ADMM has been implemented and illustrated its practical performance by Li et al. \cite{lxzone}. Nevertheless,
the convergence of the directly extended ADMM can not be theoretically guaranteed.
To address this issue, Li et al. \cite{lxzone} characterized the solutions of (\ref{dual2})
in terms of fixed-point of a proximity related operator and developed
an algorithm named 2SFPPA which is demonstrated to be very efficient  for a CS-MRI reconstruction problem.

\section{Solving problem (\ref{dual2}) by sGS technique based ADMM and generalized ADMM}\label{algsec}
We observe that the dual model (\ref{dual2}) contains three blocks of variables and two blocks of non-smooth convex functions, and hence
it can be solved by the methods of sGS technique based ADMM (abbr. sGS-ADMM) and generalized ADMM (abbr. sGS-ADMM\_G) with convergence guaranteed.

\subsection{Solving problem (\ref{dual2}) by sGS-ADMM}
In this section, we target to employ the sGS-ADMM to solve (\ref{dual2}) and establish its convergence.
The augmented Lagrangian function associated with the problem (\ref{dual2}) is defined by
\begin{align*}
  \mathcal{L}_{\sigma}(x_{1}, x_{2}, x_{3}; u)&=\delta_{\mathcal{B}_{2}^{(\mu)}}(x_{1})+\delta_{\mathcal{B}_{\infty}^{(\lambda)}}(x_{2})+\langle b, x_{3}\rangle\\
  &-\langle u, B^{\top}x_{1}+W^{\top}x_{2}+K^{\top} x_{3}\rangle+\frac{\sigma}{2}\|B^{\top}x_{1}+W^{\top}x_{2}+
  K^{\top}x_{3}\|^{2},
\end{align*}
where $\sigma>0$ is a penalty parameter and $u\in\mathbb{R}^{d}$ is a multiplier. It is well-known that,
starting from $(x_1^0,x_2^0,x_3^0)$, the classic augmented
Lagrangian method of multipliers solves
\begin{equation}\label{almalg}
(x_1^{k+1},x_2^{k+1},x_3^{k+1})=\text{arg}\min_{x_1,x_2,x_3}\mathcal{L}_{\sigma}(x_1,x_2,x_3;u^k)
\end{equation}
at the current iteration and updates the multiplier $u^{k+1}$
subsequently. Solving (\ref{almalg}) for $x_1$, $x_2$ and $x_3$ simultaneously
can be difficult, since it ignores the favorable separable structure
emerging in the objective function and the constraints. Alternatively, one may try to replace (\ref{almalg}) by directly
extended ADMM with the Gauss-Seidel order that $x_1\rightarrow x_2\rightarrow x_3\rightarrow u$. However,
the type of the approach may diverge in theory, although it often performs much
better numerically. To deal with this challenge,  we apply the intelligent sGS technique \cite{lskthree},
which groups $x_{1}$ as one block and $(x_{2}, x_{3})$ as another, and takes the cycle order
$x_1\rightarrow x_3\rightarrow x_2\rightarrow x_3\rightarrow u$ instead of the usual $x_1\rightarrow x_2\rightarrow x_3\rightarrow u$.
More precisely, with the given $(x_{1}^{k}, x_{2}^{k}, x_{3}^{k}; u^{k})$,  the new iteration
$(x_{1}^{k+1}, x_{2}^{k+1}, x_{3}^{k+1}; u^{k+1})$ is generated via the iterative scheme:
\begin{equation}\label{mrispadmm}
\left\{
  \begin{array}{l}
   x_{1}^{k+1}=\arg\min_{x_1\in\mathbb{R}^{2d}}\big\{\mathcal{L}_{\sigma}(x_{1}, x_{2}^{k}, x_{3}^{k}; u^{k})+\frac{\sigma}{2}\| x_1-x_{1}^{k} \|^{2}_{\mathcal{S}_1}\big\},\\[3mm]
   x_{3}^{k+1/2}=\arg\min_{x_3\in\mathbb{R}^p}\big\{\mathcal{L}_{\sigma}(x_{1}^{k+1}, x_{2}^{k}, x_{3}; u^{k})+\frac{\sigma}{2}\| x_3-x_{3}^{k} \|^{2}_{\mathcal{S}_3}\big\},\\[3mm]
   x_{2}^{k+1}=\arg\min_{x_2\in\mathbb{R}^q}\big\{\mathcal{L}_{\sigma}(x_{1}^{k+1}, x_{2}, x_{3}^{k+1/2}; u^{k})+\frac{\sigma}{2}\| x_2-x_{2}^{k} \|^{2}_{\mathcal{S}_2}\big\},\\[3mm]
   x_{3}^{k+1}=\arg\min_{x_3\in\mathbb{R}^p}\big\{\mathcal{L}_{\sigma}(x_{1}^{k+1}, x_{2}^{k+1}, x_{3}; u^{k})+\frac{\sigma}{2}\| x_3-x_{3}^{k} \|^{2}_{\mathcal{S}_3}\big\},\\[3mm]
   u^{k+1}=u^{k}-\tau\sigma(B^{\top}x^{k+1}_{1}+W^{\top}x^{k+1}_{2}+K^{\top}x^{k+1}_{3}),
  \end{array}
\right.
\end{equation}
where $\mathcal{S}_1=(\tau_1 I_{2d}-BB^{\top})$, $\mathcal{S}_2= (\tau_2 I_{q}-WW^{\top})$ and
$\mathcal{S}_3= (\tau_3 I_{p}-KK^{\top})$ are positive semi-definite linear operators with some appropriate choices
of $\tau_1$, $\tau_2$ and $\tau_3$. As can be seen from the framework that an extra preparation step to compute $x_{3}^{k+1/2}$ is performed
before computing $x_{2}^{k+1}$. We will show in the next subsection that this extra step can be done at moderate cost, so
that the iterative process can be performed cheaply.

We now show that the iterative manner $x_3^{k+1/2}\rightarrow x_2^{k+1}\rightarrow x_3^{k+1}$ can be grouped together as one block
$(x_2^{k+1},x_3^{k+1})$ with a specially designed semi-proximal term. The fact is directly followed from the well-known
sGS decomposition theorem of Li et al. \cite{li39} which can be stated as follows.

\begin{lemma}\label{lemma1}
For $k\geq0$, the $x_{2}$- and $x_{3}$- subproblems in (\ref{mrispadmm}) can be summarized as the following compact form:
\begin{equation}\label{x23}
  (x^{k+1}_{2},x^{k+1}_{3}) = \arg\min_{x_{2},x_{3}}\bigg\{\mathcal{L}_{\sigma}(x^{k+1}_{1}, x_{2}, x_{3}; u^{k})
  +\frac{\sigma}{2}\Big\|
  \Big(
   \begin{array}{c}
   x_{2}\\
  x_{3}\\
   \end{array}
   \Big)-
   \Big(
   \begin{array}{c}
   x^{k}_{2}\\
  x^{k}_{3}\\
   \end{array}
   \Big)
   \Big\|^{2}_{\mathcal{G}}\bigg\},
\end{equation}
where $\mathcal{G}$ is self-adjoint positive semi-definite linear operator.
\end{lemma}
\begin{proof}
Let
\begin{equation*}
  \mathcal{H}=\mathcal{D}+\mathcal{M}+\mathcal{M}^{*},
\end{equation*}
where
\begin{equation*}
\mathcal{M}=
  \left(
   \begin{array}{cc}
   0&WK^{\top}\\
   0&0\\
   \end{array}
   \right),
  \qquad
 \mathcal{D}=
  \left(
   \begin{array}{cc}
   WW^{\top}&0\\
   0&KK^{\top}\\
   \end{array}
   \right),
    \qquad
  \mathcal{M}^{*}=
  \left(
   \begin{array}{cc}
   0&0\\
   \mathcal{K} W^{\top}&0\\
   \end{array}
   \right).
\end{equation*}
Furthermore, denote
 \begin{equation*}
   \mathcal{G}_{1}=\mathcal{M}D^{-1}\mathcal{M}^{*}=  \left(
   \begin{array}{cc}
 WK^{\top}(KK^{\top})^{-1}K W^{\top}&0\\
   0&0\\
   \end{array}
   \right),
 \end{equation*}
and
 \begin{equation*}
   \mathcal{G}_{2}=
   \left(
   \begin{array}{cc}
  \mathcal{S}_2&0 \\
   0 &\mathcal{S}_3\\
   \end{array}
   \right)
   =
   \left(
   \begin{array}{cc}
 \tau_{2}I_{q}-WW^{\top}&0\\
   0&\tau_{3}I_{p}-KK^{\top}\\
   \end{array}
   \right),
 \end{equation*}
then the desired conclusion is followed from the sGS decomposition theorem in \cite[Theorem 1]{li39} by setting $\mathcal{G}=\mathcal{G}_1+\mathcal{G}_2$.
\end{proof}

Based on the result, we can rewrite (\ref{mrispadmm}) equivalently as follows
\begin{equation}\label{mriadmm2}
\left\{
  \begin{array}{l}
   x_{1}^{k+1}=\arg\min_{x_1\in\mathbb{R}^{2d}}\big\{\mathcal{L}_{\sigma}(x_{1}, x_{2}^{k}, x_{3}^{k}; u^{k})+\frac{\sigma}{2}\| x-x_{1}^{k} \|^{2}_{\mathcal{S}_1}\big\},\\[2mm]
(x^{k+1}_{2},x^{k+1}_{3}) = \arg\min_{x_{2},x_{3}}\bigg\{\mathcal{L}_{\sigma}(x^{k+1}_{1}, x_{2}, x_{3}; u^{k})
  +\frac{\sigma}{2}\Big\|
  \Big(
   \begin{array}{c}
   x_{2}\\
  x_{3}\\
   \end{array}
   \Big)-
   \Big(
   \begin{array}{c}
   x^{k}_{2}\\
  x^{k}_{3}\\
   \end{array}
   \Big)
   \Big\|^{2}_{\mathcal{G}}\bigg\},\\[2mm]
   u^{k+1}=u^{k}-\tau\sigma(B^{\top}x^{k+1}_{1}+W^{\top}x^{k+1}_{2}+K^{\top}x^{k+1}_{3}),
  \end{array}
\right.
\end{equation}
which reduces to the two-block semi-proximal ADMM.
This equivalence is very important because the convergence can be easily followed by using
the known convergence result \cite{mtfour}.
To conclude this subsection, we present the convergence result of sGS-ADMM for solving (\ref{dual2}).

\begin{theorem}(\cite[Theorem B.1]{mtfour})\label{theo1}
Suppose that the sequence ${(x^{k}_{1},x^{k}_{2},x^{k}_{3};u^{k})}$ is generated by the iterative scheme (\ref{mrispadmm}) from
an initial point $(x_1^0,x_2^0,x_3^0;u^0)$.
If $\tau\in(0, (1+\sqrt{5})/2)$ and $\mathcal{S}_1$, $\mathcal{S}_2$, and $\mathcal{S}_3$ are positive semi-definite linear operators,
the sequence $\{(x^{k}_{1},x^{k}_{2},x^{k}_{3})\}$ converges to an optimal
solution of the dual problem (\ref{dual2}) and $\{u^{k}\}$ converges to an optimal solution of the primal problem (\ref{mripro}).
\end{theorem}

\begin{remark}
One issue yet remained to be addressed is to choose the positive semi-definite linear operators $\mathcal{S}_1$, $\mathcal{S}_2$ and $\mathcal{S}_3$. From
the theory in numerical algebra, we know that this goal can be achieved by $\tau_1\geq\rho( B B^\top)$, $\tau_2\geq\rho(WW^\top)$
and $\tau_3\geq\rho(KK^\top)$, respectively, where $\rho(\cdot)$ denotes the spectral radius of a matrix.
\end{remark}

\subsection{Implement details}

Observing that each step of the iterative scheme (\ref{mrispadmm}) involves solving a convex minimization problem, we
now illustrate that a simple closed-form solution is permitted for each subproblem, which leads to the
framework easy to implement. Firstly, we can get for every $k=0, 1, ...$ that
\begin{align*}
   x_{1}^{k+1}&=\arg\min_{x_1\in\mathbb{R}^{2d}}\Big\{\mathcal{L}_{\sigma}(x_{1}, x_{2}^{k}, x_{3}^{k}; u^{k})+\frac{\sigma}{2}\| x_1-x_{1}^{k} \|^{2}_{\mathcal{S}_1}\Big\}\\[2mm]
              &=\Pi_{\mathcal{B}_2^{(\mu)}}\Big(x_{1}^{k}-\frac{1}{\tau_{1}}B\big(B^{\top}x_{1}^{k}
               +W^{\top}x_{2}^{k}+K^{\top}x_{3}^{k}-\frac{1}{\sigma}u^{k}\big)\Big),
\end{align*}
where the last equality is from (\ref{proj}). Secondly, for every $k=0, 1, ...$, we have
\begin{align*}
   x_{3}^{k+1/2}&=\arg\min_{x_3\in\mathbb{R}^p}\Big\{\mathcal{L}_{\sigma}(x_{1}^{k+1}, x_{2}^{k+1}, x_{3}; u^{k})+\frac{\sigma}{2}\| x_3-x_{3}^{k} \|^{2}_{\mathcal{S}_3}\Big\}\\[3mm]
               &=x_{3}^{k}-\frac{1}{\tau_{3}}K(B^{\top}x_{1}^{k+1}
              +W^{\top}x_{2}^{k}+K^{\top}x_{3}^{k}-\frac{1}{\sigma}u^{k})-\frac{1}{\tau_{3}\sigma}b.
\end{align*}
We note that the linear operator $KK^\top$ makes $x_3$ is located in a $p$-dimensional subspace rather than the $d$-dimensional subspace if
the primal problem (\ref{mripro}) is targeted.
Thirdly, for every $k=0, 1, ...$, the solution $x_{2}^{k+1}$ can be obtained by
\begin{align*}
   x_{2}^{k+1}&=\arg\min_{x_2\in\mathbb{R}^q}\Big\{\mathcal{L}_{\sigma}(x_{1}^{k+1}, x_{2}, x_{3}^{k+1/2}; u^{k})+\frac{\sigma}{2}\| x_2-x_{2}^{k} \|^{2}_{\mathcal{S}_2}\Big\}\\[3mm]
              &=\Pi_{\mathcal{B}_{\infty}^{(\lambda)}}\Big(x_{2}^{k}-\frac{1}{\tau_{2}}W\big(B^{\top}x_{1}^{k+1}
                +W^{\top}x_{2}^{k}+K^{\top}x_{3}^{k+1/2}-\frac{1}{\sigma}u^{k}\big)\Big).
\end{align*}
The derivation processes indicate that each subproblem enjoys analytic solution provided by properly choosing the semi-proximal terms with
respect to $\mathcal{S}_1$, $\mathcal{S}_2$ and $\mathcal{S}_3$.

In summary, we are ready to state the full steps of the sGS-ADMM
while it is used to solve the dual model problem (\ref{dual2})  as follows:

\begin{algorithm}
\noindent
{\bf Algorithm: sGS-ADMM}
\vskip 1.0mm \hrule \vskip 1mm
\noindent
\textbf{Step 0.} Choose starting point $(x_1^0,x_2^0,x_3^0;u^0)$. Choose positive constants $\tau_i$ such that $\mathcal{S}_i$ ($i=1,2,3$) are positive semi-definite.
Choose positive constants $\sigma>0$ and $\tau\in(0,(1+\sqrt{5})/2)$. For $k=0,1,\ldots$, do the following operations iteratively.\\
\textbf{Step 1.} Given $x_1^k$, $x_2^k$, $x_3^k$, and $u^k$, compute
$$
x_1^{k+1}=\Pi_{\mathcal{B}_2^{(\mu)}}\Big(x_{1}^{k}-\frac{1}{\tau_{1}}B\big(B^{\top}x_{1}^{k}
               +W^{\top}x_{2}^{k}+K^{\top}x_{3}^{k}-\frac{1}{\sigma}u^{k}\big)\Big).
$$
\textbf{Step 2.} Given $x_1^{k+1}$, $x_2^k$, $x_3^k$, and $u^k$, compute
$$
x_3^{k+1/2}=x_{3}^{k}-\frac{1}{\tau_{3}}K(B^{\top}x_{1}^{k+1}
              +W^{\top}x_{2}^{k}+K^{\top}x_{3}^{k}-\frac{1}{\sigma}u^{k})-\frac{1}{\tau_{3}\sigma}b.
$$
\textbf{Step 3.} Given $x_1^{k+1}$, $x_2^k$, $x_3^{k+1/2}$, and $u^k$, compute
$$
x_2^{k+1}=\Pi_{\mathcal{B}_{\infty}^{(\lambda)}}\Big(x_{2}^{k}-\frac{1}{\tau_{2}}W\big(B^{\top}x_{1}^{k+1}
                +W^{\top}x_{2}^{k}+K^{\top}x_{3}^{k+1/2}-\frac{1}{\sigma}u^{k}\big)\Big).
$$
\textbf{Step 4.} Given $x_1^{k+1}$, $x_2^{k+1}$, $x_3^{k+1/2}$, and $u^k$, compute
$$
x_3^{k+1}=x_{3}^{k}-\frac{1}{\tau_{3}}K\Big(B^{\top}x_{1}^{k+1}
              +W^{\top}x_{2}^{k+1}+K^{\top}x_{3}^{k+1/2}-\frac{1}{\sigma}u^{k}\Big)-\frac{1}{\tau_{3}\sigma}b.
$$
\textbf{Step 5.} Given $x_1^{k+1}$, $x_2^{k+1}$, $x_3^{k+1}$, and $u^k$, compute
$$
u^{k+1}=u^{k}-\tau\sigma(B^{\top}x^{k+1}_{1}+W^{\top}x^{k+1}_{2}+K^{\top}x^{k+1}_{3}).
$$
\end{algorithm}

\subsection{Solving problem (\ref{dual2}) by generalized ADMM}\label{gadmmsec}
This subsection is devoted to applying the  generalized ADMM
of Xiao et al.\cite{xiao40} to solve problem (\ref{dual2}). Again, we view variable $x_{1}$
as one group and view $(x_{2}, x_{3})$ as another, and use the sGS techniqueis
with order $x_{3}\rightarrow x_{2} \rightarrow x_{3}$ in the second group. Then the  generalize ADMM  obeys the following form while it is used to solve
the dual problem (\ref{dual2}), that is
\begin{equation}\label{gspadmm}
\left\{
  \begin{array}{l}
   x_{1}^{k+1}=\arg\min_{x_1\in\mathbb{R}^{2d}}\big\{\mathcal{L}_{\sigma}(x_{1}, \tilde{x}_{2}^{k}, \tilde{x}_{3}^{k}; \tilde{u}^{k})+\frac{\sigma}{2}\| x_1-\tilde{x}_{1}^{k} \|^{2}_{\mathcal{S}_1}\big\},\\[3mm]
   u^{k+1}=\tilde{u}^{k}-\sigma(B^{\top}x^{k+1}_{1}+W^{\top}\tilde{x}^{k}_{2}+K^{\top}\tilde{x}^{k}_{3}),\\[3mm]
   x_{3}^{k+1/2}=\arg\min_{x_3\in\mathbb{R}^p}\big\{\mathcal{L}_{\sigma}(x_{1}^{k+1}, \tilde{x}_{2}^{k}, x_{3}; u^{k+1})+\frac{\sigma}{2}\| x_3-\tilde{x}_{3}^{k} \|^{2}_{\mathcal{S}_3}\big\},\\[3mm]
   x_{2}^{k+1}=\arg\min_{x_2\in\mathbb{R}^q}\big\{\mathcal{L}_{\sigma}(x_{1}^{k+1}, x_{2}, x_{3}^{k+1/2}; u^{k+1})+\frac{\sigma}{2}\| x_2-\tilde{x}_{2}^{k} \|^{2}_{\mathcal{S}_2}\big\},\\[3mm]
   x_{3}^{k+1}=\arg\min_{x_3\in\mathbb{R}^p}\big\{\mathcal{L}_{\sigma}(x_{1}^{k+1}, x_{2}^{k+1}, x_{3}; u^{k+1})+\frac{\sigma}{2}\| x_3-\tilde{x}_{3}^{k} \|^{2}_{\mathcal{S}_3}\big\},\\[3mm]
   \tilde{\omega}^{k+1}=\tilde{\omega}^k+\rho(\omega^{k+1}-\tilde{\omega}^k),
  \end{array}
\right.
\end{equation}
where $\omega=(x_{1}, x_{2}, x_{3}, u)$, $\rho\in(0,2)$, and $\mathcal{S}_1$, $\mathcal{S}_2$ and $\mathcal{S}_3$ are positive semi-definite linear operators defined before.
We note that if $\rho=0$ and $\mathcal{S}_{i}=1$ for $i=1,2,3$, the scheme (\ref{gspadmm})
reduces to (\ref{mrispadmm}) with unite steplength $\tau=1$.
By mimicking the proof of Lemma \ref{lemma1}, we can obtain that the sGS iteration with order $x_3\rightarrow x_2\rightarrow x_3$ can be
viewed together with an additional semi-proximal term based on linear operator. Therefore, the global convergence
of corresponding algorithm for (\ref{gspadmm}) can be obtained from the \cite[Theorem 5.1]{xiao40}, which can be stated as follows without proof.
\begin{theorem}(\cite[Theorem 5.1]{xiao40})\label{theo2}
Suppose that the sequence ${(x^{k}_{1},x^{k}_{2},x^{k}_{3};u^{k})}$ is generated by the iterative scheme (\ref{gspadmm}) from
an initial point $(x_1^0,x_2^0,x_3^0;u^0)$.
If $\rho\in(0, 2)$ and $\mathcal{S}_1$, $\mathcal{S}_2$, and $\mathcal{S}_3$ are positive semi-definite linear operators,
the sequence $\{(x^{k}_{1},x^{k}_{2},x^{k}_{3})\}$ converges to an optimal
solution of the dual problem (\ref{dual2}) and $\{u^{k}\}$ converges to an optimal solution of the primal problem (\ref{mripro}).
\end{theorem}

To end this subsection, we list the full steps of sGS-ADMM\_G while it is used to solve the dual
problem (\ref{dual2}).

\begin{algorithm}
\noindent
{\bf Algorithm: sGS-ADMM\_G}
\vskip 1.0mm \hrule \vskip 1mm
\noindent
\textbf{Step 0.} Choose starting point $(x_1^0,x_2^0,x_3^0;u^0)$. Choose positive constants $\tau_i$ such that $\mathcal{S}_i$ ($i=1,2,3$) are positive semi-definite.
Choose positive constants $\sigma>0$ and $\rho\in(0,2)$. For $k=0,1,\ldots$, do the following operations iteratively.\\
\textbf{Step 1.} Given $\tilde{x}_1^k$, $\tilde{x}_2^k$, $\tilde{x}_3^k$, and $\tilde{u}^k$, compute
$$
x_1^{k+1}=\Pi_{\mathcal{B}_2^{(\mu)}}\Big(\tilde{x}_{1}^{k}-\frac{1}{\tau_{1}}B\big(B^{\top}\tilde{x}_{1}^{k}
               +W^{\top}\tilde{x}_{2}^{k}+K^{\top}\tilde{x}_{3}^{k}-\frac{1}{\sigma}\tilde{u}^{k}\big)\Big).
$$
\textbf{Step 2.} Given $x_1^{k+1}$, $\tilde{x}_2^{k}$, $\tilde{x}_3^{k}$, and $\tilde{u}^k$, compute
$$
u^{k+1}=\tilde{u}^{k}-\sigma(B^{\top}x^{k+1}_{1}+W^{\top}\tilde{x}^{k}_{2}+K^{\top}\tilde{x}^{k}_{3}).
$$
\textbf{Step 3.} Given $x_1^{k+1}$, $\tilde{x}_2^k$, $\tilde{x}_3^k$, and $u^{k+1}$, compute
$$
x_3^{k+1/2}=\tilde{x}_{3}^{k}-\frac{1}{\tau_{3}}K\Big(B^{\top}x_{1}^{k+1}
              +W^{\top}\tilde{x}_{2}^{k}+K^{\top}\tilde{x}_{3}^{k}-\frac{1}{\sigma}u^{k+1}\Big)-\frac{1}{\tau_{3}\sigma}b.
$$
\textbf{Step 4.} Given $x_1^{k+1}$, $\tilde{x}_2^k$, $x_3^{k+1/2}$, and $u^{k+1}$, compute
$$
x_2^{k+1}=\Pi_{\mathcal{B}_{\infty}^{(\lambda)}}\Big(\tilde{x}_{2}^{k}-\frac{1}{\tau_{2}}W\big(B^{\top}x_{1}^{k+1}
                +W^{\top}\tilde{x}_{2}^{k}+K^{\top}x_{3}^{k+1/2}-\frac{1}{\sigma}u^{k+1}\big)\Big).
$$
\textbf{Step 5.} Given $x_1^{k+1}$, $x_2^{k+1}$, $\tilde{x}_3^{k}$, and $u^k$, compute
$$
x_3^{k+1}=\tilde{x}_{3}^{k}-\frac{1}{\tau_{3}}K\Big(B^{\top}x_{1}^{k+1}
              +W^{\top}x_{2}^{k+1}+K^{\top}\tilde{x}_{3}^{k}-\frac{1}{\sigma}u^{k+1}\Big)-\frac{1}{\tau_{3}\sigma}b.
$$
\textbf{Step 6.} Given $x_1^{k+1}$, $x_2^{k+1}$, $x_3^{k+1}$, and $u^{k+1}$, compute
$$
\tilde{\omega}^{k+1}=\tilde{\omega}^k+\rho(\omega^{k+1}-\tilde{\omega}^k).
$$
\end{algorithm}

\section{Numerical experiments}\label{numsec}

In this section, we construct a series of numerical experiments to evaluate the practical performance of sGS-ADMM
and sGS-ADMM\_G against the state-of-the-art algorithm 2SFPPA. All the experiments are performed with Microsoft Windows 10 and MATLAB
R2018a, and run on a PC with an Intel Core i7 CPU at 1.80 GHz and 8 GB of memory.

\subsection{General descriptions}
First of all, some descriptions should be given for the following series of tests. We conduct experiments on some typical MRI data: a ``Shepp-Logan" phantom,
some  brain images. We do the reconstruction from the retrospectively
undersampled Fourier coefficients of these images.
The k-space undersampling is simulated by using the following masks: Cartesian sampling with random phase
encoding \cite{ldfive}, 2D random sampling \cite{ldfive,qg,x50}, pseudo radial
sampling \cite{trj}.
The Haar wavelet transform $W\in\mathbb{R}^{q\times d}$ is chosen to be
non-decimated and thus we have that $q=4d$. We set parameters in objective function
 be same as the reference \cite{lxzone}. Accordingly,
we set the diagonal entries of the diagonal matrix $\Lambda$ as follows
\begin{equation*}
 \lambda_{i}=\left\{
\begin{array}{ll}
  0,\ & i\in\{1, 2, \ldots, d\}, \\[2mm]
  \frac{1}{2},\ & i\in\{d+1, d+2, \ldots, q\},
\end{array}
\right.
\end{equation*}
and take the regularization parameter in (\ref{mripro}) as $\mu=3$. For the 2SFPPA, we compile the code according to
the Algorithm $1$ in \cite{lxzone}, and set the parameters $\alpha_{1}=1/8$ and $\alpha_{2}= \alpha_{3}=0.9$, which have been illustrated to be able to achieve the best results in most cases for each dataset. For the other algorithms,
we fix $\tau_{1}=8$, and $\tau_{2}=\tau_{3}=10/9$. Besides, the steplength $\tau$ involved in sGS-ADMM  is set to
be the thumb value $1.618$, and the relaxation factor in sGS-ADMM\_G is set as $\rho=1.4$.

Here, we use the peak signal-to-noise ratio (PSNR) in the unit of dB to measure the quality of the restored images, i.e.,$$\mbox{PSNR}:=10\log_{10}\frac{255 \sqrt{d}}{\|\bar{u}-u\|}(\mbox{dB}).$$
Moreover, different from \cite{z51,x50},  we adopt a couple of rigorous criteria in evaluating the validity of the involved algorithms. First, we use the relative $\ell_2$-norm error (RLNE) defined as $$\text{RLNE}:=\frac{\|\bar{u}-u\|}{\|u\|},$$
where $\bar{u}$ is the ground truth image and $u$ is the reconstructed image. Second, based on KKT system (\ref{KKTnew}), we use the following KKT residuals criterion defined as
$$
\text{RelErr}:=\max\{\eta_{P}, \eta_{D}, \eta_1, \eta_2 \},
$$
where
\begin{equation*}
\left\{
  \begin{array}{l}
   \eta_{P}:=\dfrac{\|K u-b\|}{1+\|b\|},\\[3mm]
   \eta_D:=\|B^{\top}x_{1}+W^{\top}x_{2}+K^{\top}x_{3}\|,\\[3mm]
   \eta_{1}:=\dfrac{\|x_1-\Pi_{\mathcal{B}_2^{(\mu)}}(x_1+Bu)\|}{1+\|x_1\|+\|Bu\|},\\[3mm]
   \eta_{2}:=\dfrac{\|x_2-\Pi_{\mathcal{B}_{\infty}^{(\lambda)}}(x_2+Wu)\|}{1+\|x_2\|+\|Wu\|}.
  \end{array}
\right.
\end{equation*}
In order to balance the primal and dual infeasibilities for accelerating the iteration of both presented algorithms, we use the variable penalty parameters strategies defined
as follows. Moreover, 2SFPPA also use this updating technique for comparisons in a fair way.
We initialize $\sigma_0=5e-3$ (resp. $\beta_0$ in 2SFPPA) and update iteratively via the following form:
$$
\sigma_{k+1}=\left\{
\begin{array}{lll}
  \min\{1.25\sigma_{k}, 10^{-2}\},\ & \mbox{if} \ \eta_{P}/\eta_{D}\leq\ 1/5, \\[1mm]
  \max\{0.8\sigma_{k}, 10^{-5}\},\ & \mbox{if} \ \eta_{P}/\eta_{D}\geq 5, \\[1mm]
  \sigma_{k},\ & \mbox{otherwise}.
\end{array}
\right.
$$
The iterative process of each algorithm is
terminated if RelErr or RLNE is sufficiently small,
or the maximum
iteration number is achieved.
We tried different starting points for each algorithm and found that all of them are insensitive
towards starting points. Therefore, we initialize $(x_1^0, x_2^0, x_3^0, u^0)$ as zero in all experiments of the following.

\subsection{Experiments on phantom data}

In this subsection, we test the efficiency of each algorithm using a simple ``Shepp-Logan" phantom image with size $256\times 256$ as shown at
leftmost plot in Figure \ref{fig2}. The sampling pattern is simulated by using the pseudo radial sampling mask with sampling rate $6.5\%$ which is displayed
at the second plot in Figure \ref{fig2}.  In this test, we consider the noiseless case, i.e., $\epsilon=0$ in  (\ref{orig}).

Firstly, we compare the reconstruction qualities obtained by each algorithm within $100$ iterations. The recovery images produced by 2SFPPA, sGS-ADMM
and sGS-ADMM\_G are listed from the third to the last plot in Figure \ref{fig2}, respectively.  As can be observed from the
last three images that, all the algorithms produced acceptable reconstructions within so few number of iterations.
To be precise, 2SFPPA obtained an image with RLNE $=7.21\%$, which is slightly larger than $2.38\%$ and $2.17\%$ derived by other two algorithms.
To further visibly illustrate the superiority of the ADMMs, we draw the $10$ times scaled difference images of (c-e) in Figure \ref{fig2} to the true
image (a). The compared heat images are listed in Figure \ref{fig22}, respectively. We observe that the reconstructed images of sGS-ADMM and sGS-ADMM\_G
exhibit obvious benefits  compared with 2SFPPA.

\begin{figure}[ht]
\begin{center}
\begin{tabular}{c}
\subfigure[]{\includegraphics[width=0.19\textwidth,height=0.19\textwidth]{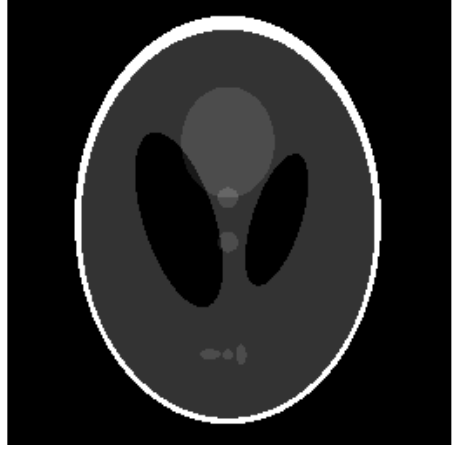}}
\subfigure[]{\includegraphics[width=0.19\textwidth,height=0.19\textwidth]{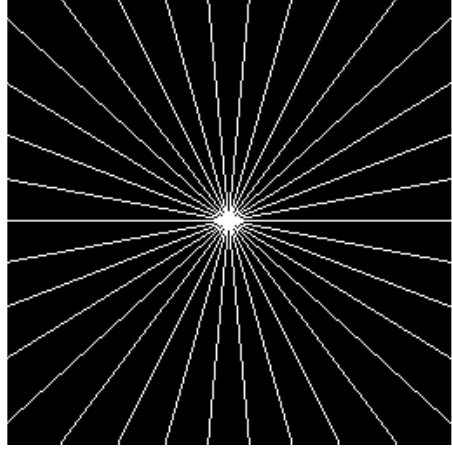}}
\subfigure[]{\includegraphics[width=0.19\textwidth,height=0.19\textwidth]{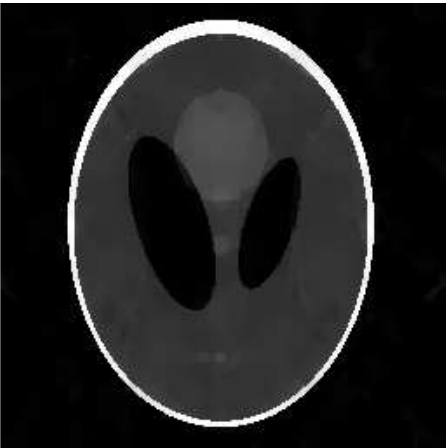}}
\subfigure[]{\includegraphics[width=0.19\textwidth,height=0.19\textwidth]{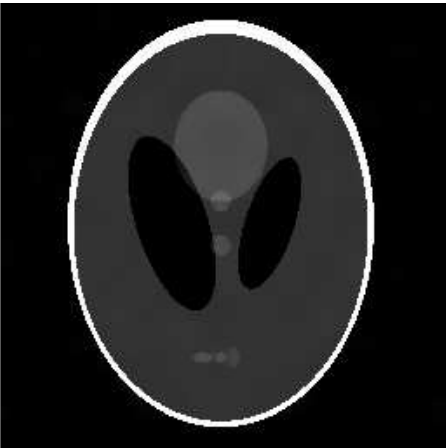}}
\subfigure[]{\includegraphics[width=0.19\textwidth,height=0.19\textwidth]{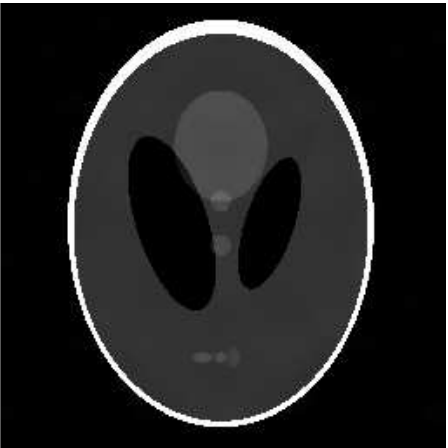}}
\end{tabular}
\end{center}
\caption{\scriptsize Reconstruction results using the phantom image: (a) the true image, (b) the pseudo radial sampling pattern of
sampling rate $6.5\%$, (c)-(e) the reconstructed images by 2SFPPA, sGS-ADMM
and sGS-ADMM\_G with the RLNE errors are $7.21\%$, $2.38\%$ and $2.17\%$ within $100$ iterations,
respectively.}
\label{fig2}
\end{figure}
\begin{figure}[ht]
\begin{center}
\begin{tabular}{c}
\subfigure[]{\includegraphics[width=0.25\textwidth,height=0.25\textwidth]{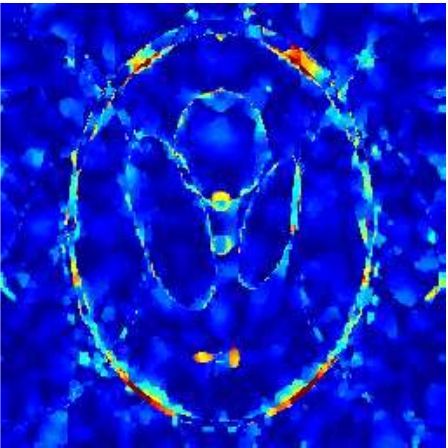}}
\subfigure[]{\includegraphics[width=0.25\textwidth,height=0.25\textwidth]{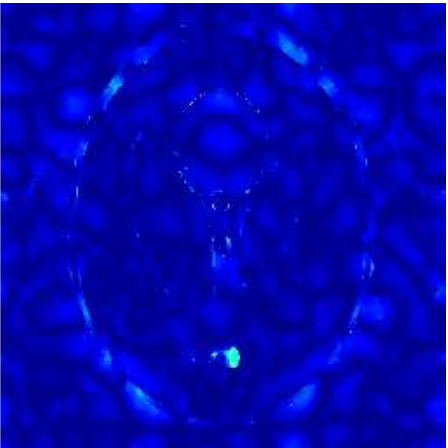}}
\subfigure[]{\includegraphics[width=0.33\textwidth,height=0.255\textwidth]{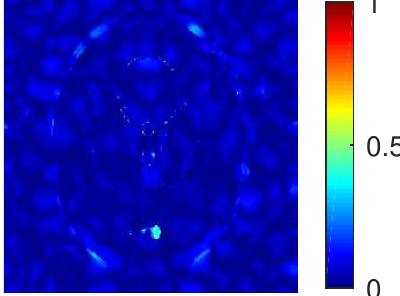}}
\end{tabular}
\end{center}
\caption{\scriptsize The $10\times$ scaled difference images of reconstructed images of 2SFPPA,  sGS-ADMM
and sGS-ADMM\_G to the ground truth image.}
\label{fig22}
\end{figure}

Secondly, we determine each algorithm be successful if the criterion ``$\text{RLNE} \leq \text{Tol}$" is achieved.
Then we investigate the numerical performance of each algorithm with different error tolerances Tol. In this test, we use the parameters values
by default as the aforementioned except that the iterative process is forced to be terminated
when maximum iterations $3000$ is exceed without achieving convergence. We report the detailed results in Table \ref{tab1}
with respect to the numbers of iterations (`Iter'), the PSNR values (`PSNR'), and the computing time in seconds (`CPU')
for ``Shepp-Logan" phantom reconstruction. In addition, the symbol ``-" indicates that the corresponding algorithm cannot achieve the given accuracy within
permissible number of iterations.
\begin{table}[htbp]
{\scriptsize\centering \caption{Numerical comparisons results for ``Shepp-Logan" phantom}
\begin{tabular}{l | ccc | ccc | ccc }
\hline
\multicolumn{1}{c|}{} & \multicolumn{3}{c|}{$\text{Tol}=10^{-2}$} & \multicolumn{3}{c|}{$\text{Tol}=10^{-3}$} & \multicolumn{3}{c}{$\text{Tol}=10^{-4}$}\\
\hline
        Algorithm  &Iter&PSNR(dB)&CPU&Iter&PSNR(dB)&CPU&Iter&PSNR(dB)&CPU\\
\hline
2SFPPA              &182 &50.2   &26.3       &1804 &60.2   &299.8  &$-$  &$-$      &$-$ \\
\hline
sGS-ADMM            &166 &50.2   &28.6       &616  &60.3   &127.9  &1754 &70.2     &291.3 \\
\hline
sGS-ADMM\_G         &135 &50.3   &20.3      &739  &60.2    &121.3  &1951 &70.2     &290.9 \\
\hline
\end{tabular}\label{tab1}
}
\end{table}
\begin{figure}[h!]
\begin{center}
\begin{tabular}{c}
\subfigure[]{\includegraphics[width=0.3\textwidth,height=0.3\textwidth]{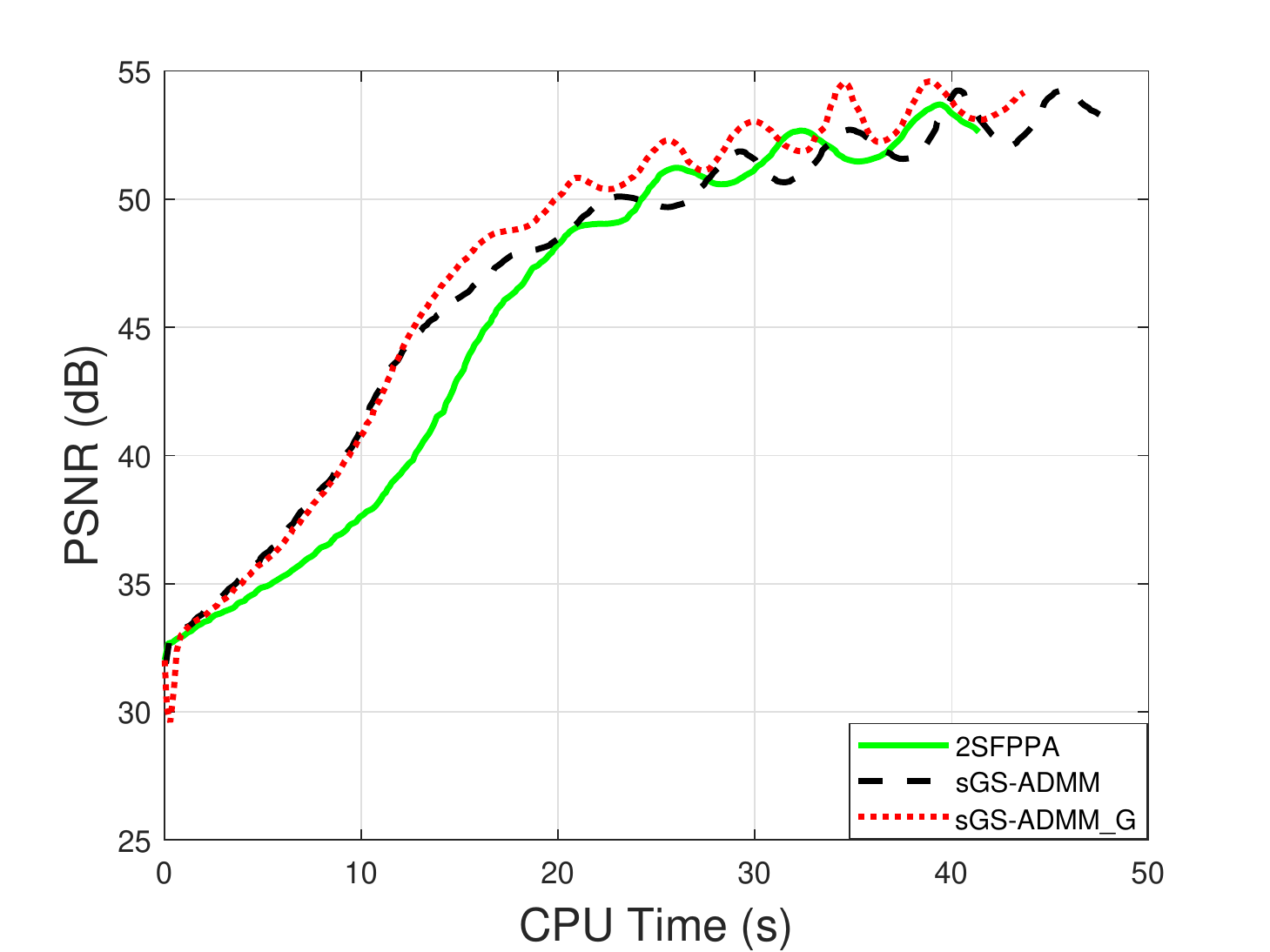}}
\subfigure[]{\includegraphics[width=0.3\textwidth,height=0.3\textwidth]{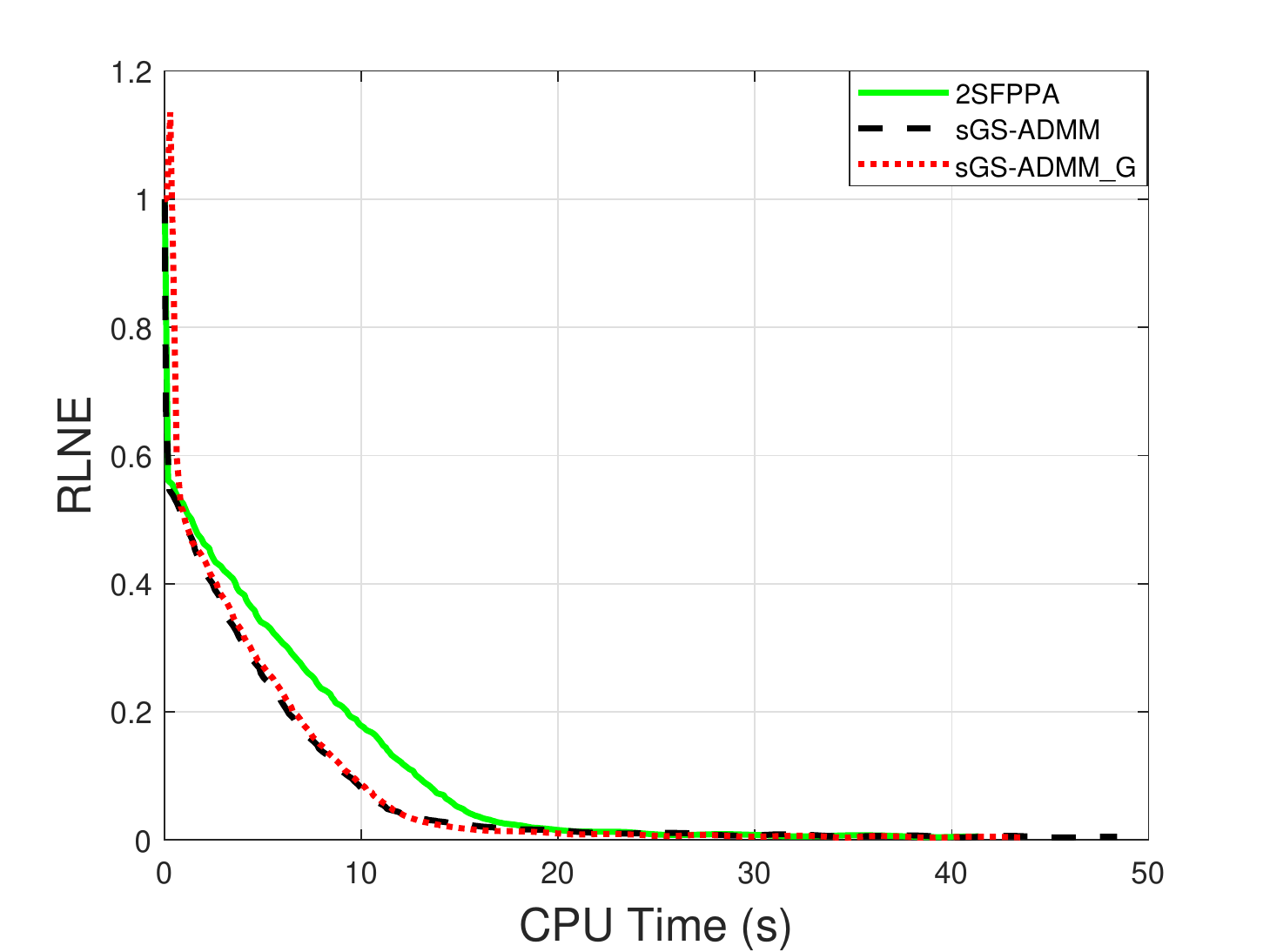}}
\subfigure[]{\includegraphics[width=0.3\textwidth,height=0.3\textwidth]{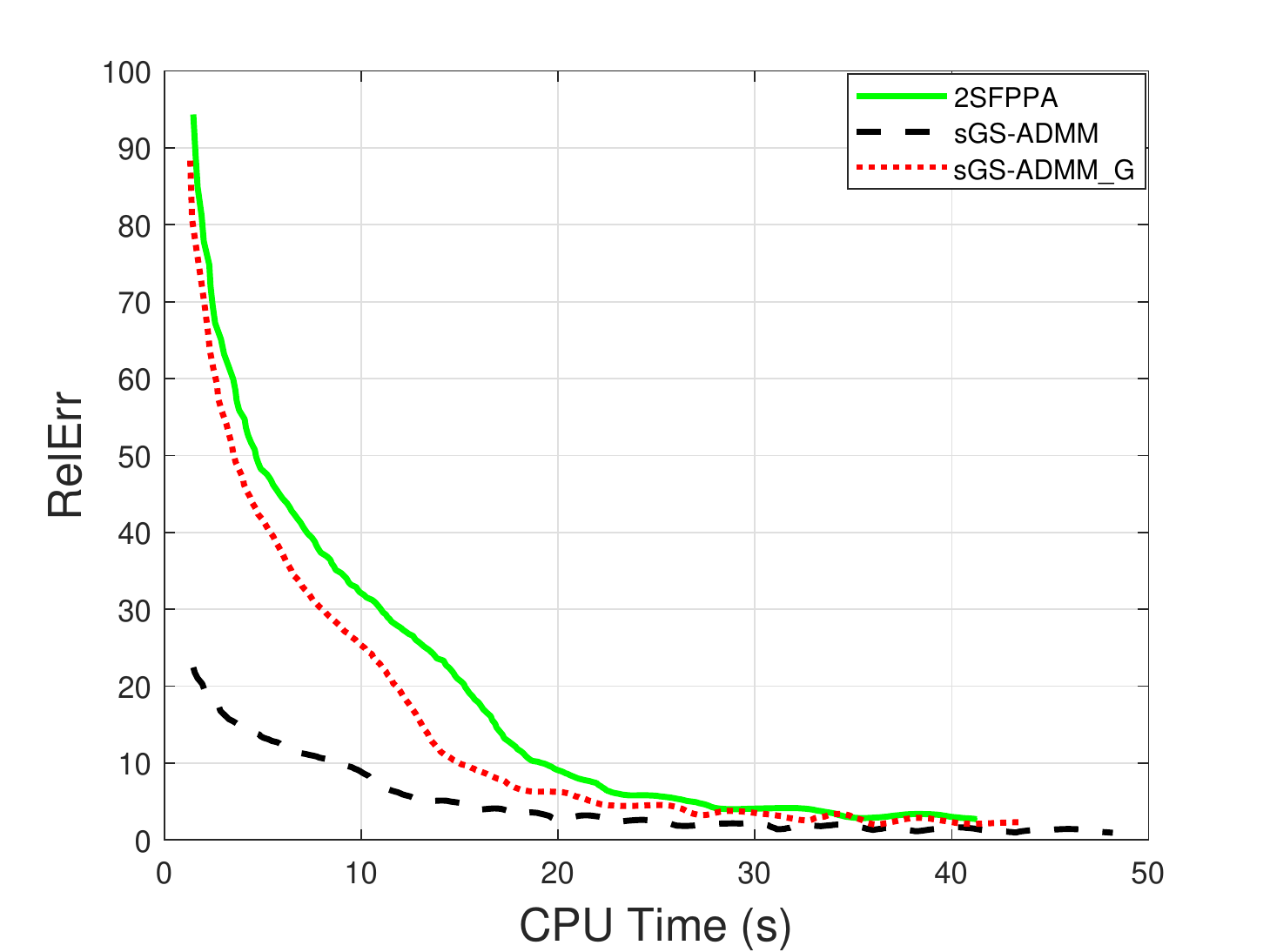}}\\
\subfigure[]{\includegraphics[width=0.3\textwidth,height=0.3\textwidth]{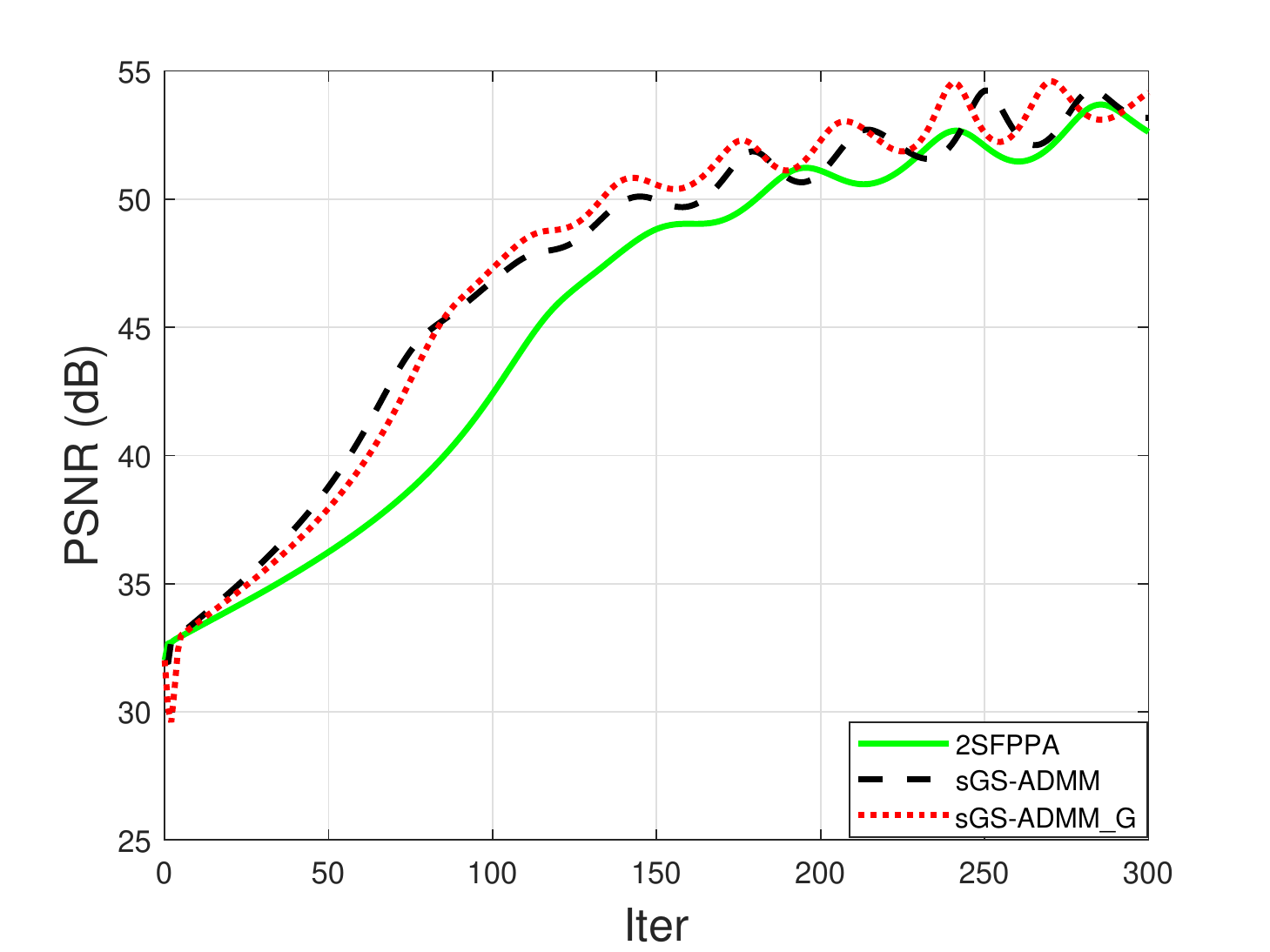}}
\subfigure[]{\includegraphics[width=0.3\textwidth,height=0.3\textwidth]{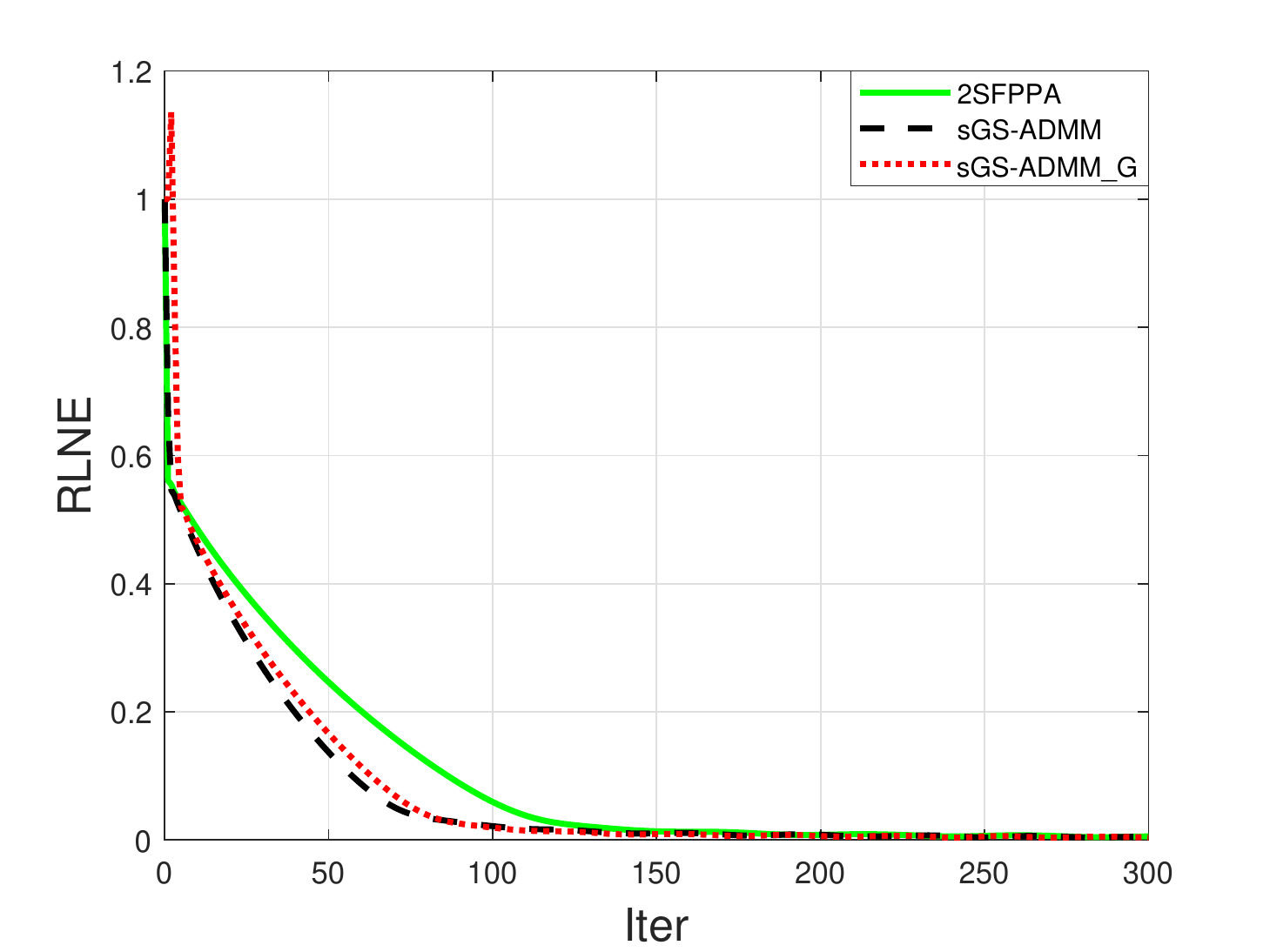}}
\subfigure[]{\includegraphics[width=0.3\textwidth,height=0.3\textwidth]{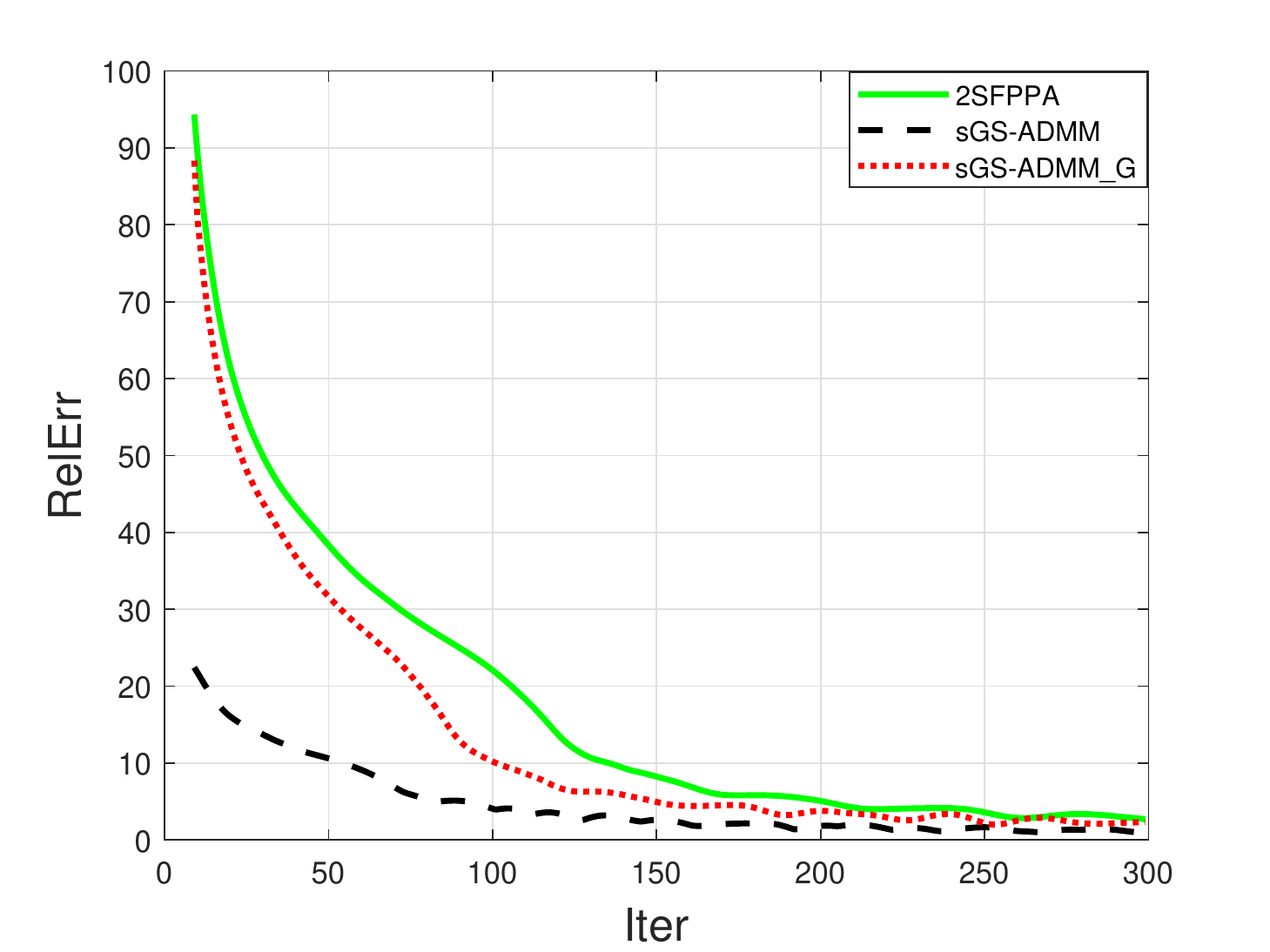}}
\end{tabular}
\end{center}
\caption{{\scriptsize Reconstruction results using the phantom image. (a) PSNR versus computational time;
(b) RLNE versus computational time;
(c) RelErr versus computational time;
(d) PSNR versus number of iterations; (e) RLNE versus number of iterations; (f) RelErr versus number of iterations.}}
\label{fig3}
\end{figure}

We can observe from this table that almost all the tested algorithms can work successfully to produce the reconstructions except
2SFPPA failed to derive higher precision solutions.
All the algorithms require more iterations and computing time with the improvement
of accuracy, and the ADMM type methods apparently perform much better than 2SFPPA.
To more visually examine the algorithms' performance, we draw the curves of PSNR, RLNE and RelErr with  respect to the iteration numbers and running time increase in Figure \ref{fig3}. It is worth mentioning that we abandoned the first 10 iterations in order to better highlight the trend of RelErr  curve in the plots of righthand side. We find that the RelErr values are always higher than the RLNE values in the iteration process, so
the stopping criteria can be chosen as $\text{RLNE}\leq\text{Tol}=5e-3$ or the maximum
iteration number 300 is achieved.
As can be seen from the two plots at the lefthand side of this figure, the curves derived by 2SFPPA are almost always at the bottom which indicate that 2SFPPA is the slowest to increase
the PSNR values. The downward trend of RLNE and RelErr shows that all the algorithms are effective in the middle and the right side plots. Moreover,
the sGS-ADMM and sGS-ADMM\_G decrease faster than that from 2SFPPA in sense of RLNE and RelErr values.
In summary, those simple tests preliminarily illustrate the efficiency of these
employed algorithms in recovering phantom images.

\subsection{Experiments on brain imaging data}

In this subsection, we further investigate the validity of sGS-ADMM and sGS-ADMM\_G.
The experiments are conducted on a couple of T2-weighted brain images where the ground truth images with real and imaginary
parts are shown in the first plots of Figures \ref{fig51} and \ref{fig52}.
The brain images to be tested are two slices acquired from a healthy
volunteer at a 3T Siemens Trio Tim MRI scanners using the T2-weighted turbo spin echo sequence.  The  Figure \ref{fig51} (a) is a $256 \times 256$ brain image
($\text{TR/TE}=6100/99 \text{ms}$, $\text{FOV}=220\times220 \text{mm}^{2}$, $\text{slice thickness}=3.0 \text{mm}$), and the Figure \ref{fig52} (a) is another $512 \times 512$ brain image
($\text{TR/TE} = 5000/97 \text{ms}, \text{FOV} = 230 \times 187 \text{mm}^2, \text{slice thickness} = 5.0 \text{mm}$). Again, the pseudo radial sampling scheme is used in these tests.
The sampling rate of  the foregoing image is chosen $18.11\%$ as shown in Figure \ref{fig51} (b), and  $9.31\%$ as shown in Figure \ref{fig52} (b), respectively.
As usual in MRI \cite{guass1,guass2}, the i.i.d. complex Gaussian noise with standard deviation $\varrho=0.005$  is added to the k-space
to verify the robustness of the algorithms. The last three plots (c-e) of Figures \ref{fig51} and \ref{fig52}  represent the restoration results of algorithms 2SFPPA, sGS-ADMM and sGS-ADMM\_G within 40 iterations for the two brain images, respectively.
\begin{figure}[ht]
\begin{center}
\begin{tabular}{c}
\subfigure[]{\includegraphics[width=0.19\textwidth,height=0.19\textwidth]{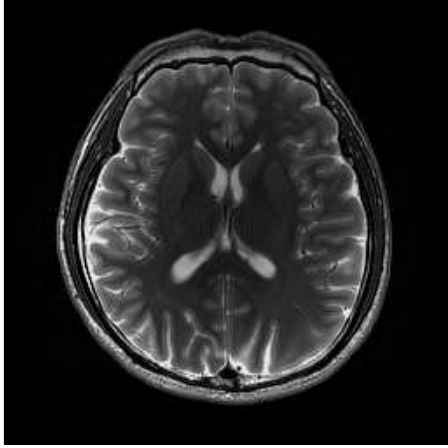}}
\subfigure[]{\includegraphics[width=0.19\textwidth,height=0.19\textwidth]{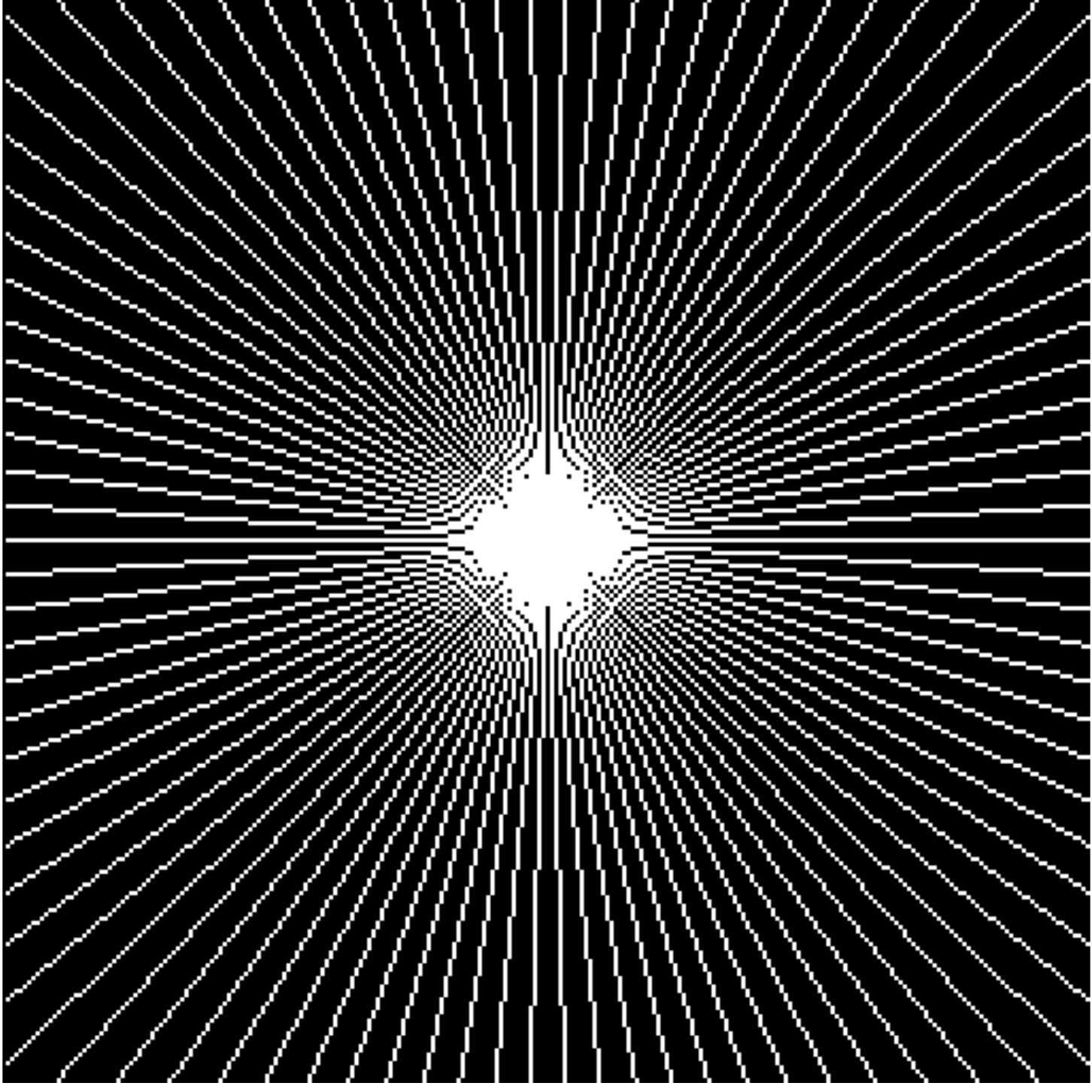}}
\subfigure[]{\includegraphics[width=0.19\textwidth,height=0.19\textwidth]{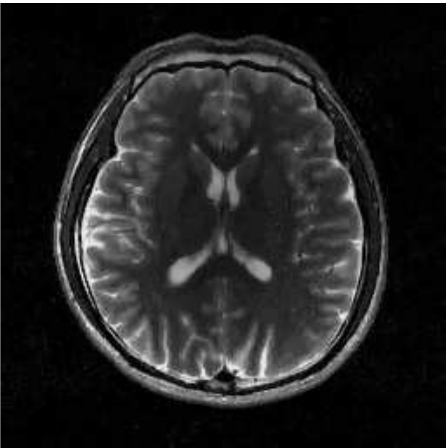}}
\subfigure[]{\includegraphics[width=0.19\textwidth,height=0.19\textwidth]{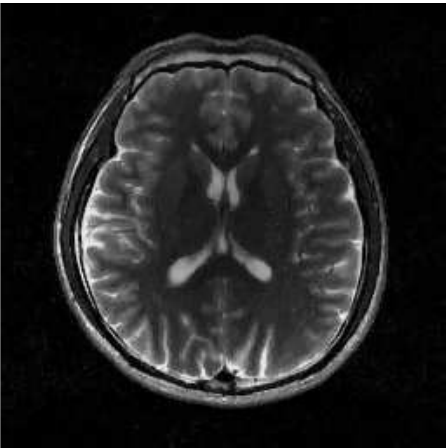}}
\subfigure[]{\includegraphics[width=0.19\textwidth,height=0.19\textwidth]{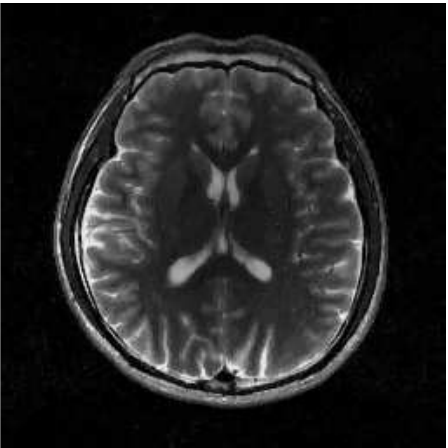}}
\end{tabular}
\end{center}
\caption{\scriptsize  Reconstruction brain images $256\times 256$ T2-brain images and using the pseudo radial sampling pattern of sampling rate 18.11\%. (a) A full sampled brain image; (b) pseudo radial sampling mask; (c-e) reconstructed images by 2SFPPA, sGS-ADMM and sGS-ADMM\_G, respectively.}
\label{fig51}
\end{figure}
\begin{figure}[ht]
\begin{center}
\begin{tabular}{c}
\subfigure[]{\includegraphics[width=0.19\textwidth,height=0.19\textwidth]{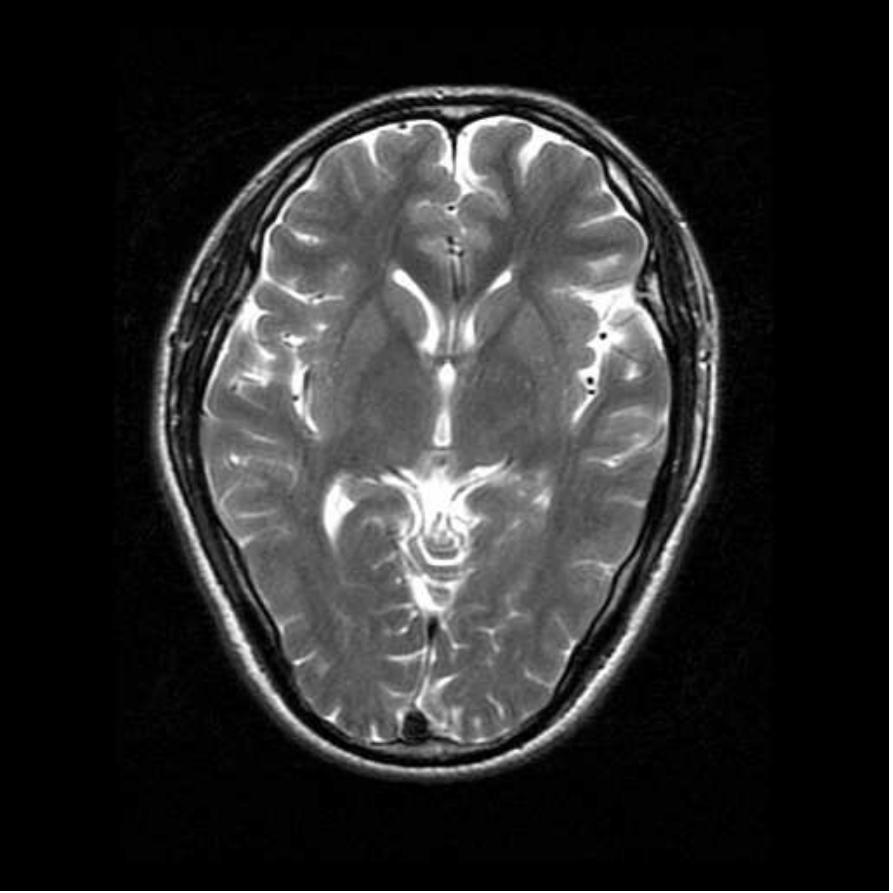}}
\subfigure[]{\includegraphics[width=0.19\textwidth,height=0.19\textwidth]{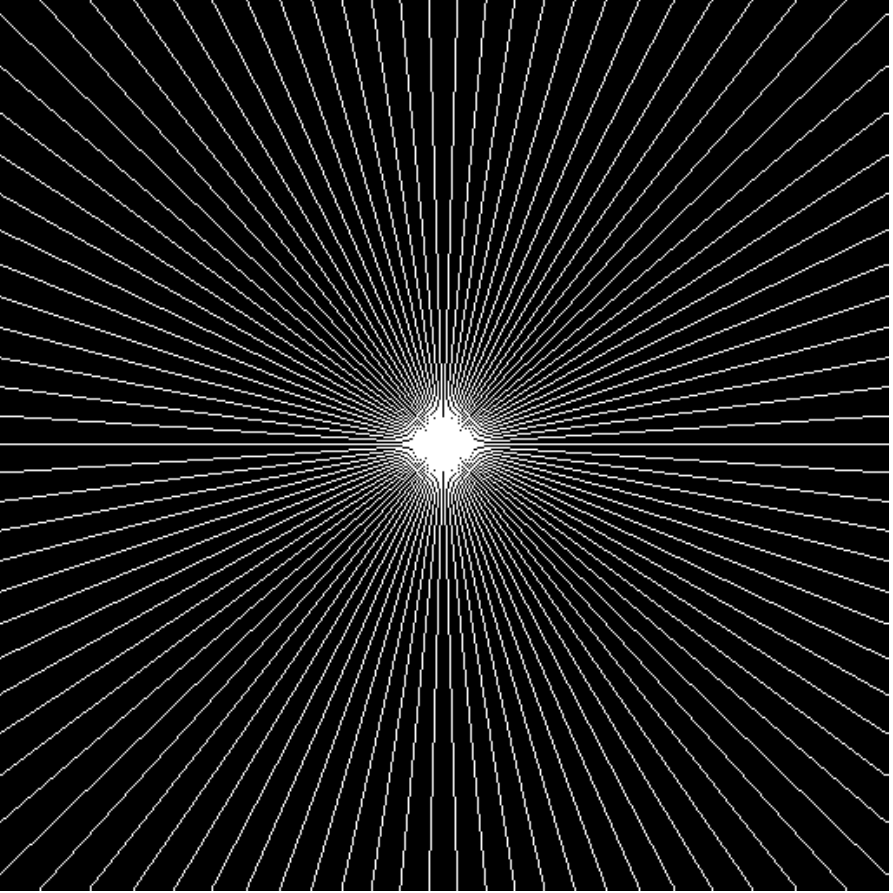}}
\subfigure[]{\includegraphics[width=0.19\textwidth,height=0.19\textwidth]{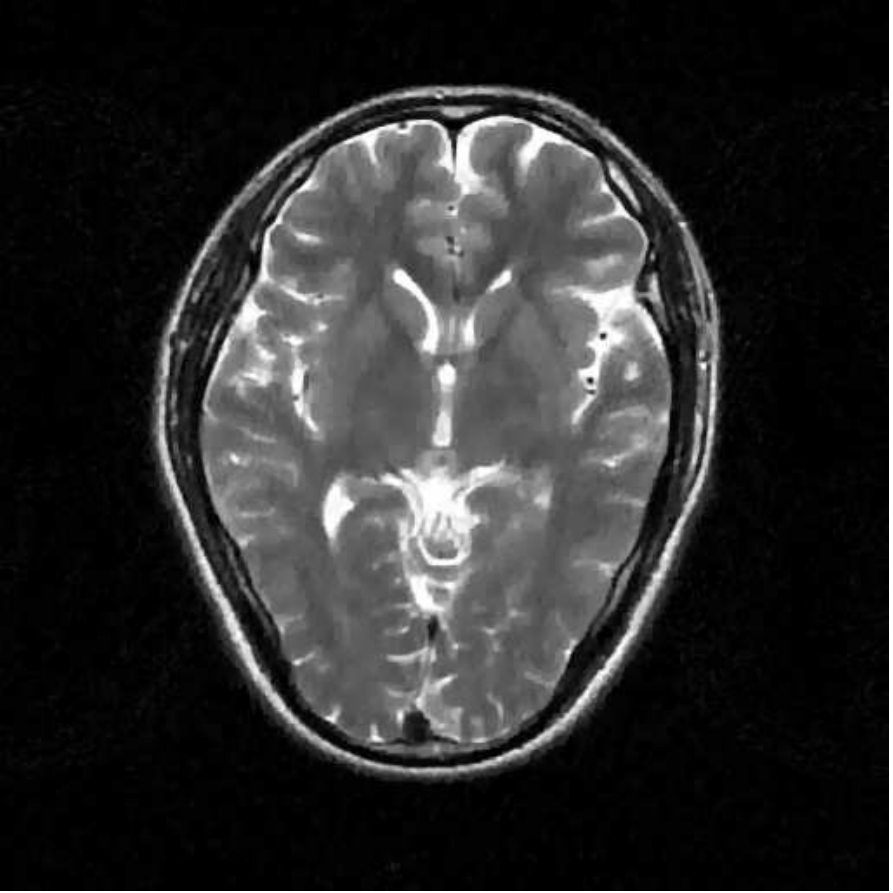}}
\subfigure[]{\includegraphics[width=0.19\textwidth,height=0.19\textwidth]{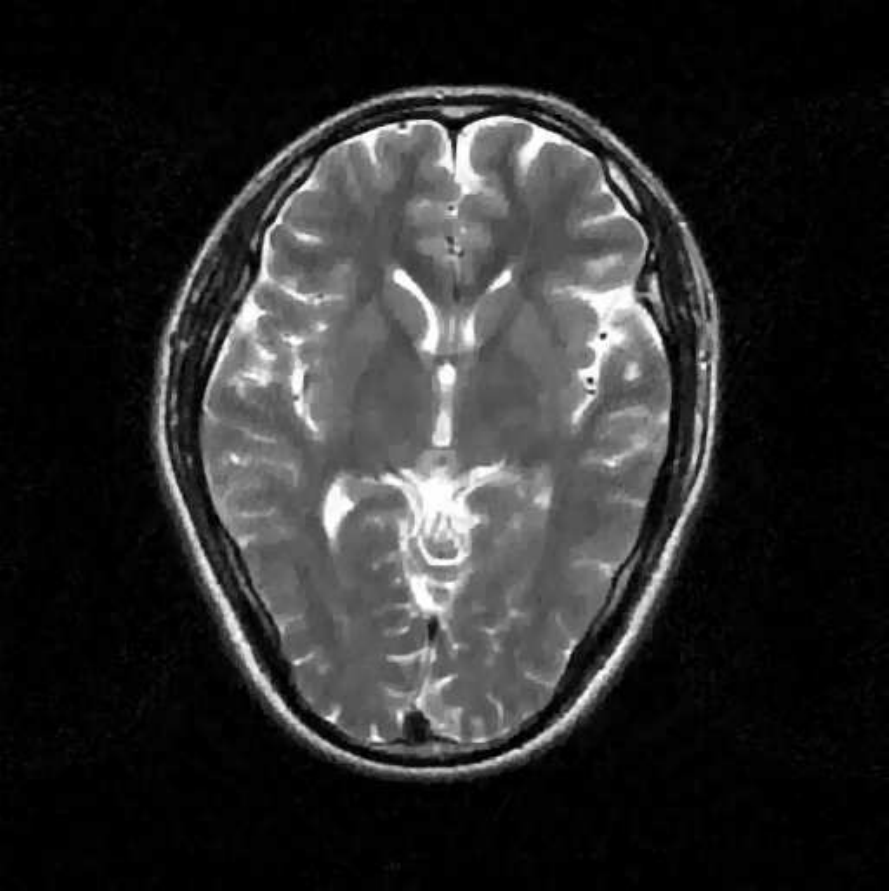}}
\subfigure[]{\includegraphics[width=0.19\textwidth,height=0.19\textwidth]{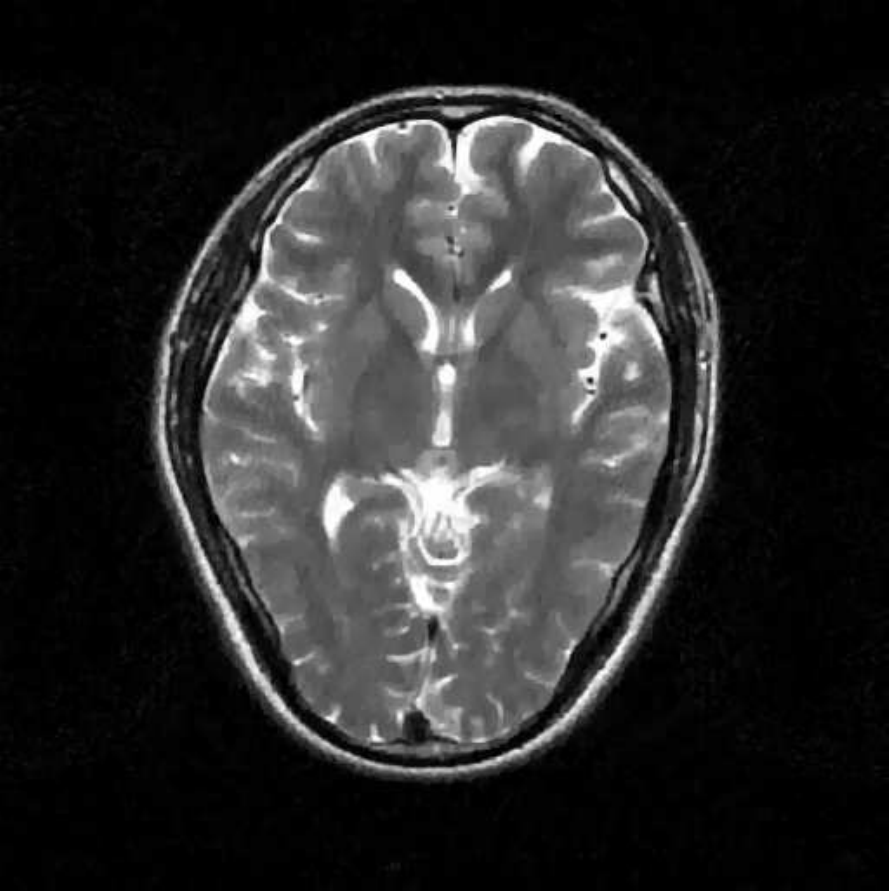}}
\end{tabular}
\end{center}
\caption{\scriptsize Reconstruction brain images $512\times 512$ T2-brain images and using the pseudo radial sampling pattern of sampling rate 9.31\%. (a) A full sampled brain image; (b) pseudo radial sampling mask; (c-e) Reconstructed images using 2SFPPA, sGS-ADMM and sGS-ADMM\_G, respectively.}
\label{fig52}
\end{figure}

We can see that all the tested algorithms achieved the acceptable reconstruction results within so few iteration steps although it is hard to see the
comparison results of the three algorithms clearly. To see the latent convergence behaviors of each algorithm, we
draw some curves regrading PSNR, RLNE and RelErr of each algorithm for each tested image with respect to the time and iteration numbers.
As before, we still abandon the previous $10$ ill-conditioned iterations here when drawing the curves about RelErr.
Details on each tested image can be found in Figures \ref{fig53} and \ref{fig54}.
As can be seen from these figures that, the results of sGS-ADMM\_G fluctuated at the beginning, but it eventually been flat.
From the PSNR curves, we can see that 2SFPPA has the fastest upward trend at the beginning, but sGS-ADMM catched up soon
and kept the advantage to the end. This phenomenon is also happed similarly in the RLNE curves.
For the curves derived by RelErr, we see that ADMM type methods
decrease slightly faster than 2SFPPA as shown in the right plots.
Based on the above explanations, we can conclude that the ADMMs are really effective and practical in rebuilding real data of different sizes.

\begin{figure}[h!]
\begin{center}
\begin{tabular}{c}
\subfigure[]{\includegraphics[width=0.3\textwidth,height=0.3\textwidth]{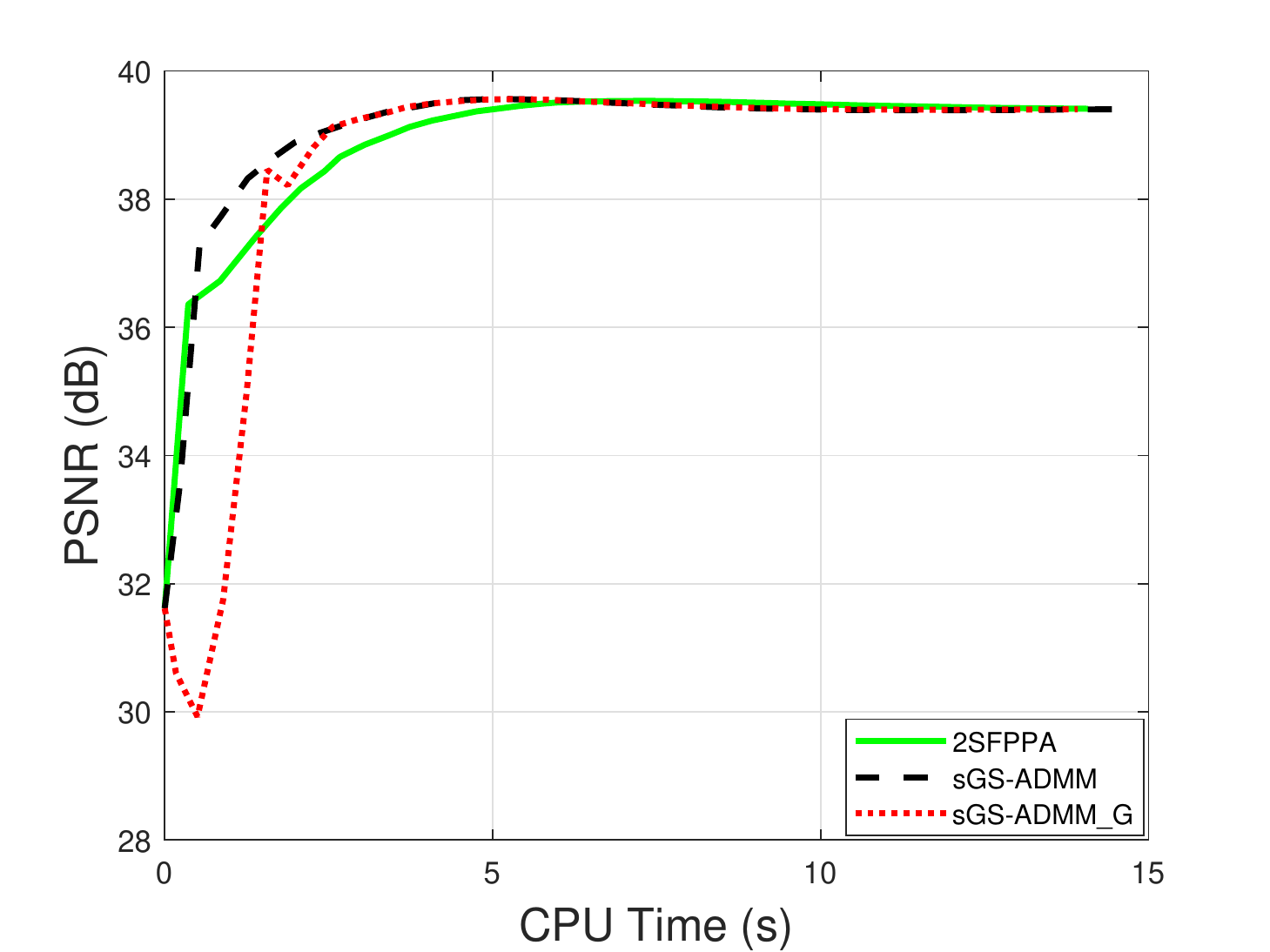}}
\subfigure[]{\includegraphics[width=0.3\textwidth,height=0.3\textwidth]{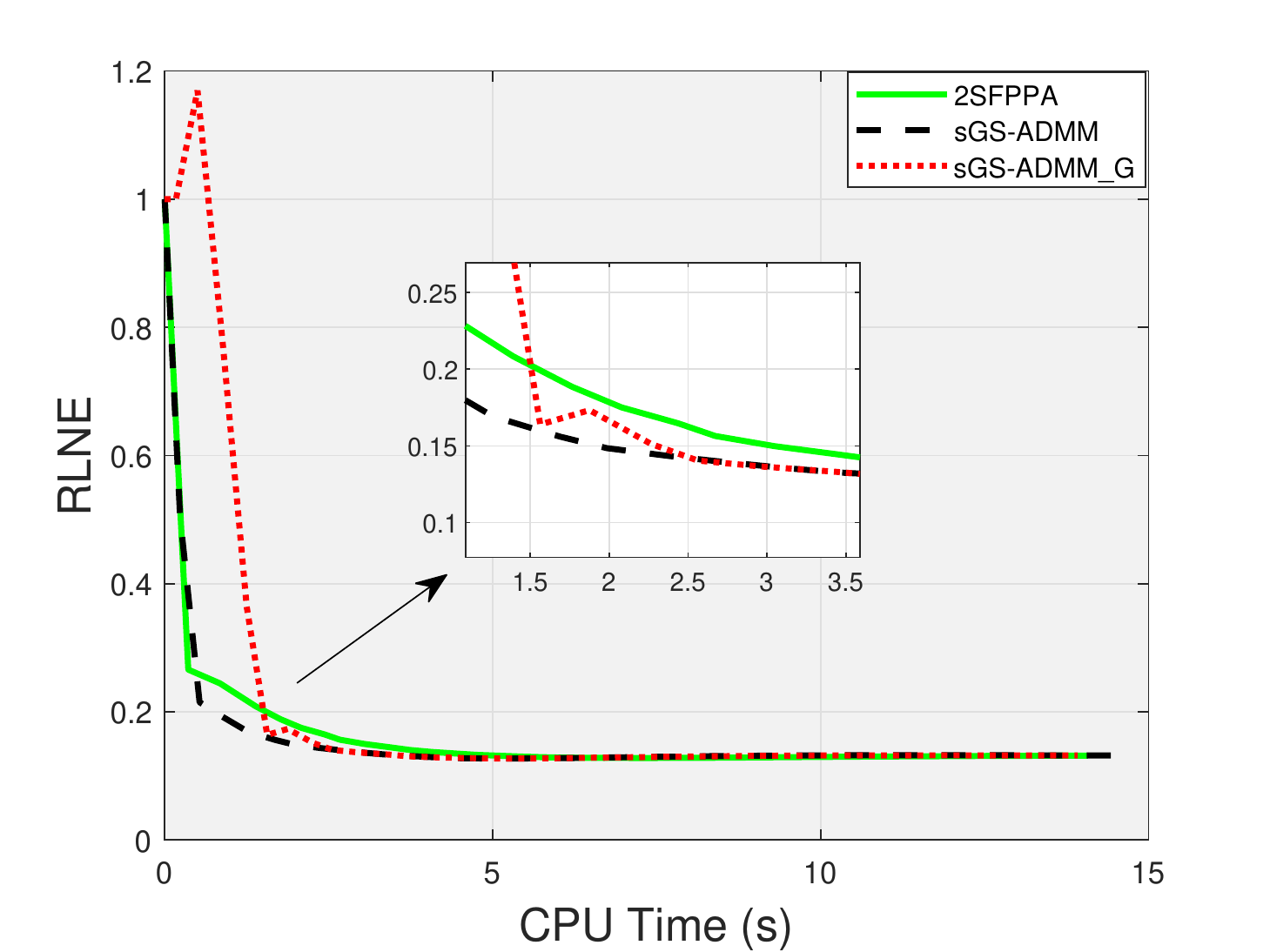}}
\subfigure[]{\includegraphics[width=0.3\textwidth,height=0.3\textwidth]{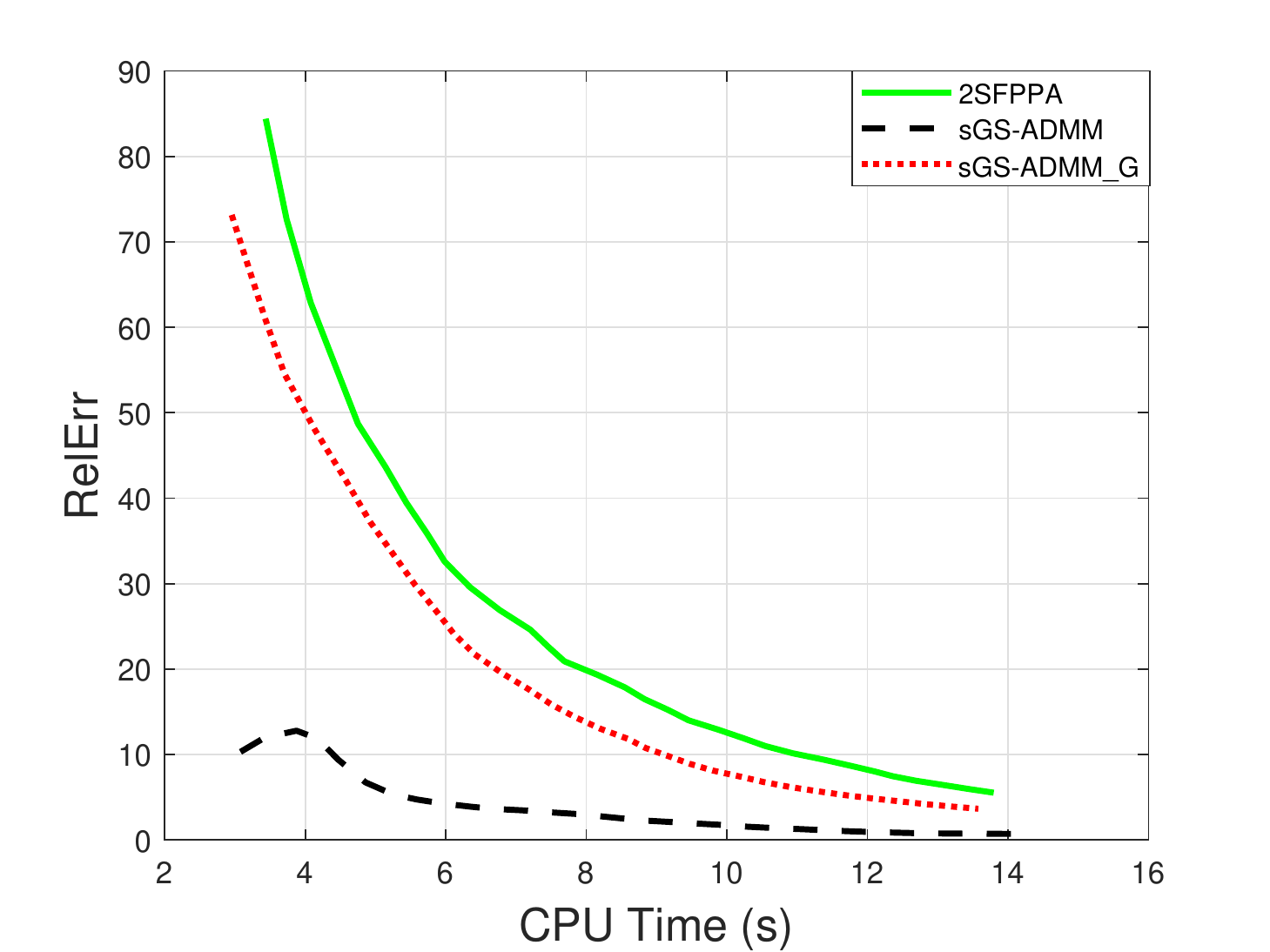}}\\
\subfigure[]{\includegraphics[width=0.3\textwidth,height=0.3\textwidth]{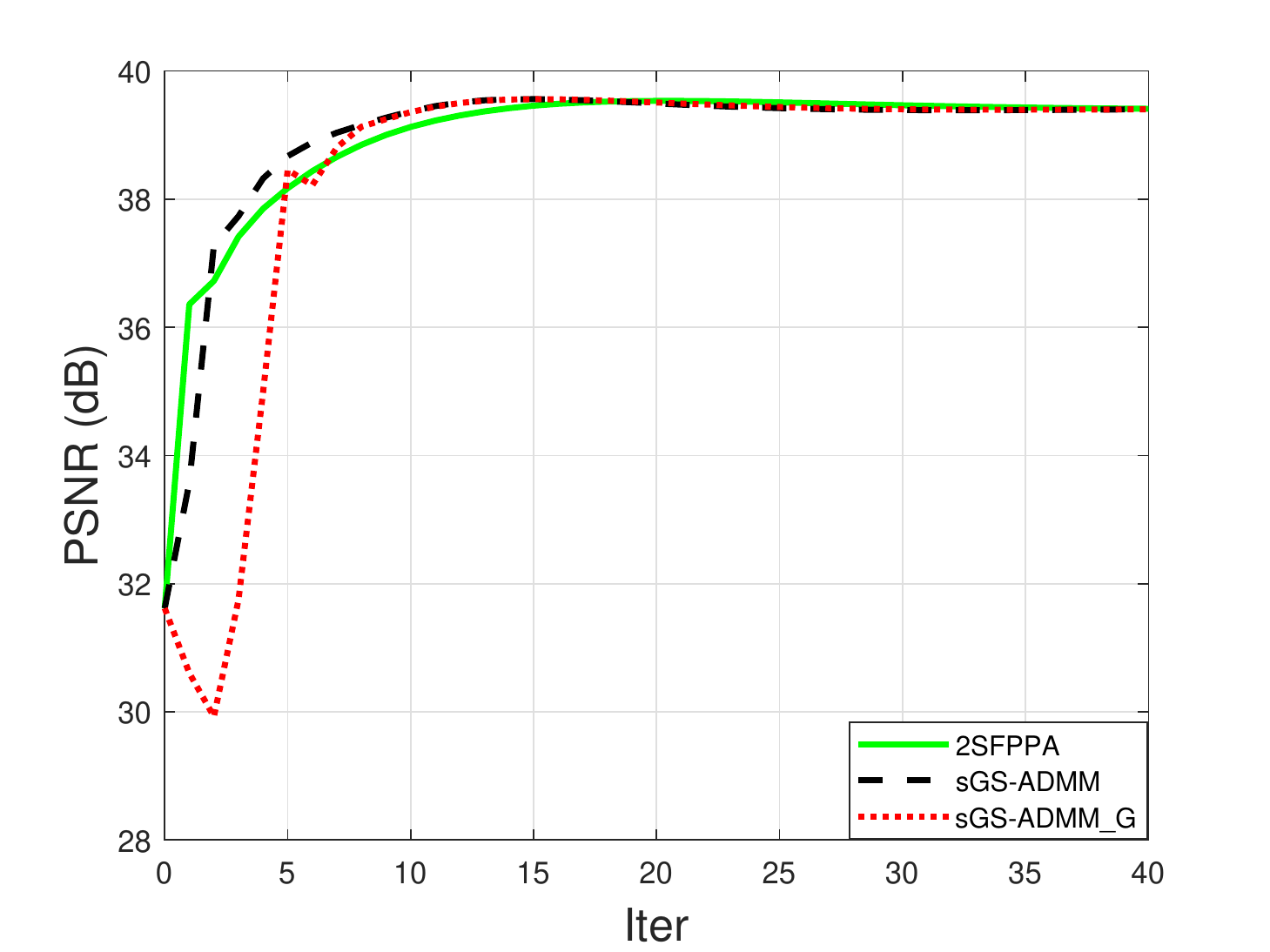}}
\subfigure[]{\includegraphics[width=0.3\textwidth,height=0.3\textwidth]{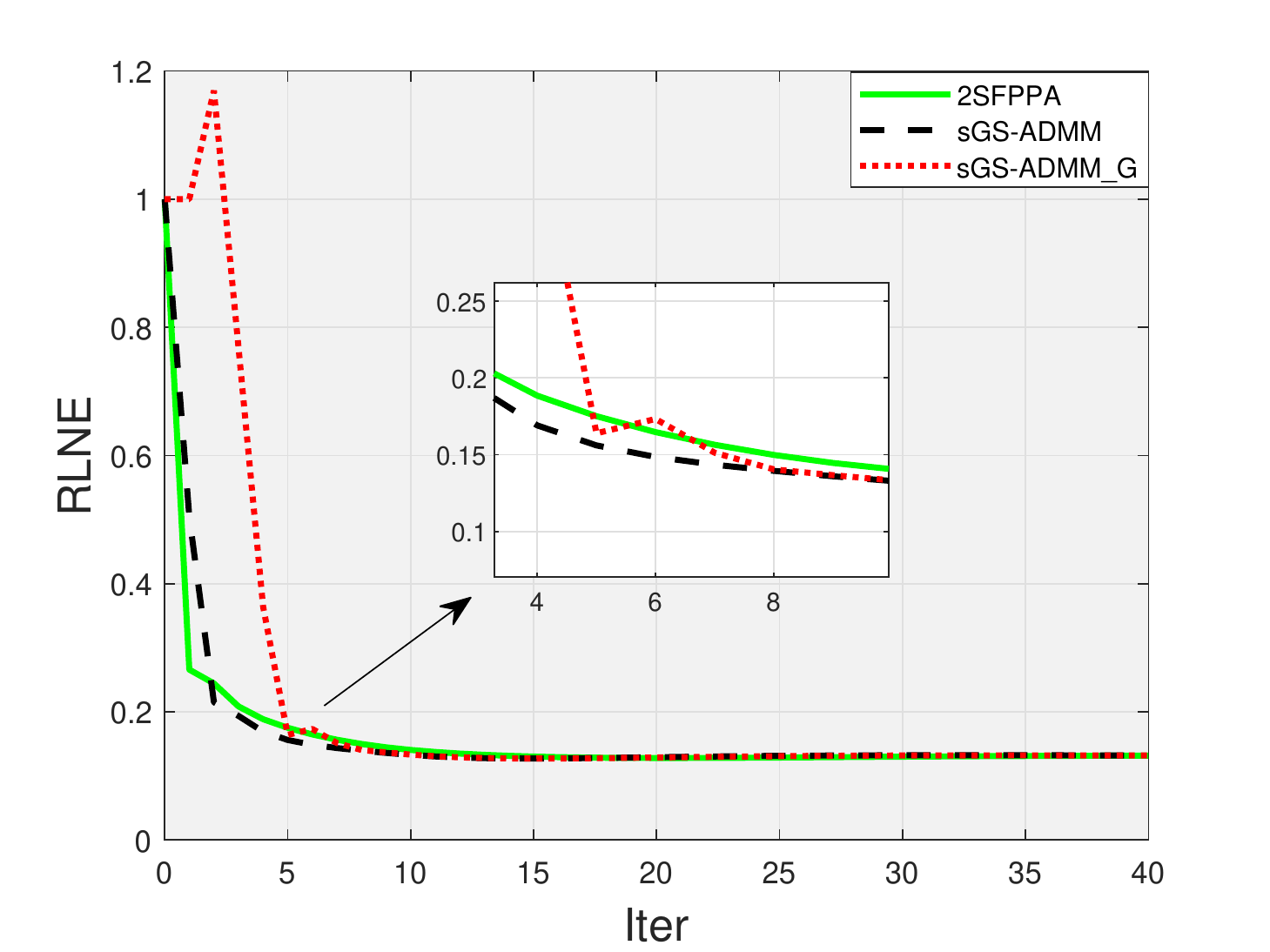}}
\subfigure[]{\includegraphics[width=0.3\textwidth,height=0.3\textwidth]{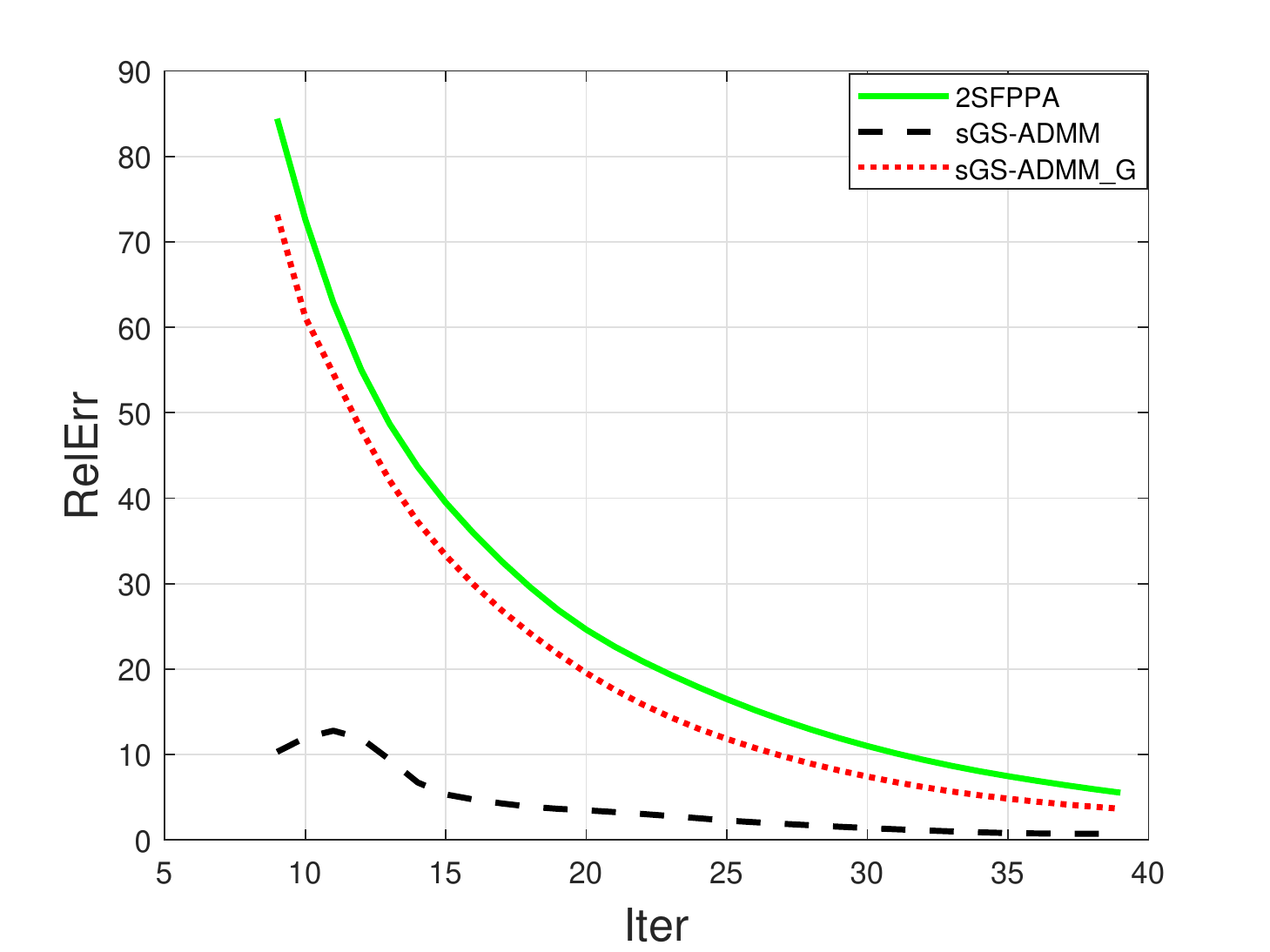}}
\end{tabular}
\end{center}
\caption{{\scriptsize Reconstruction results of the brian image with size $256\times 256$. (a) PSNR versus computational time;
(b) RLNE versus computational time; (c) RelErr versus computational time;
(d) PSNR versus number of iterations; (e) RLNE versus number of iterations;
(f) RelErr versus number of iterations.}}
\label{fig53}
\end{figure}
\begin{figure}[h!]
\begin{center}
\begin{tabular}{c}
\subfigure[]{\includegraphics[width=0.3\textwidth,height=0.3\textwidth]{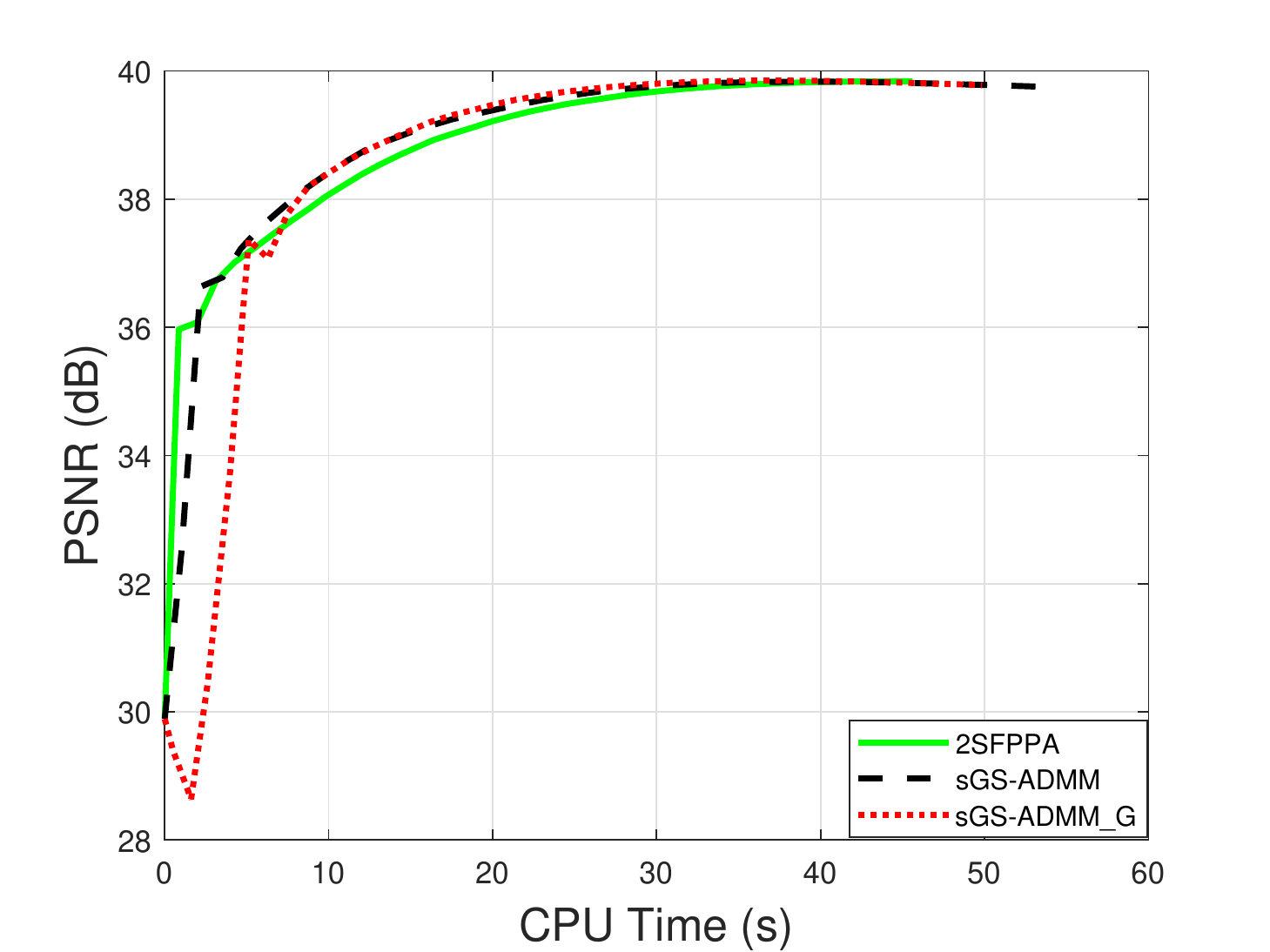}}
\subfigure[]{\includegraphics[width=0.3\textwidth,height=0.3\textwidth]{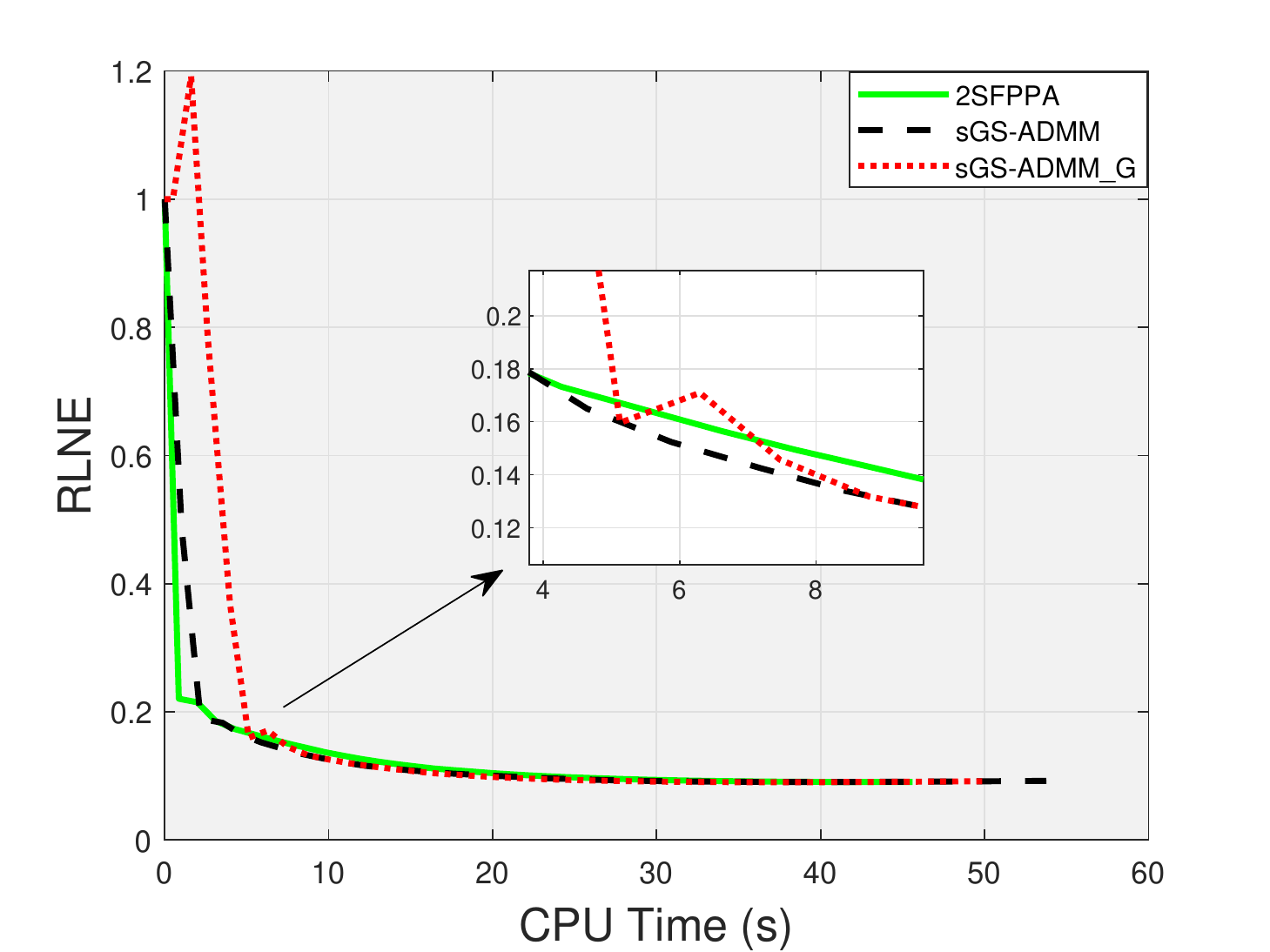}}
\subfigure[]{\includegraphics[width=0.3\textwidth,height=0.3\textwidth]{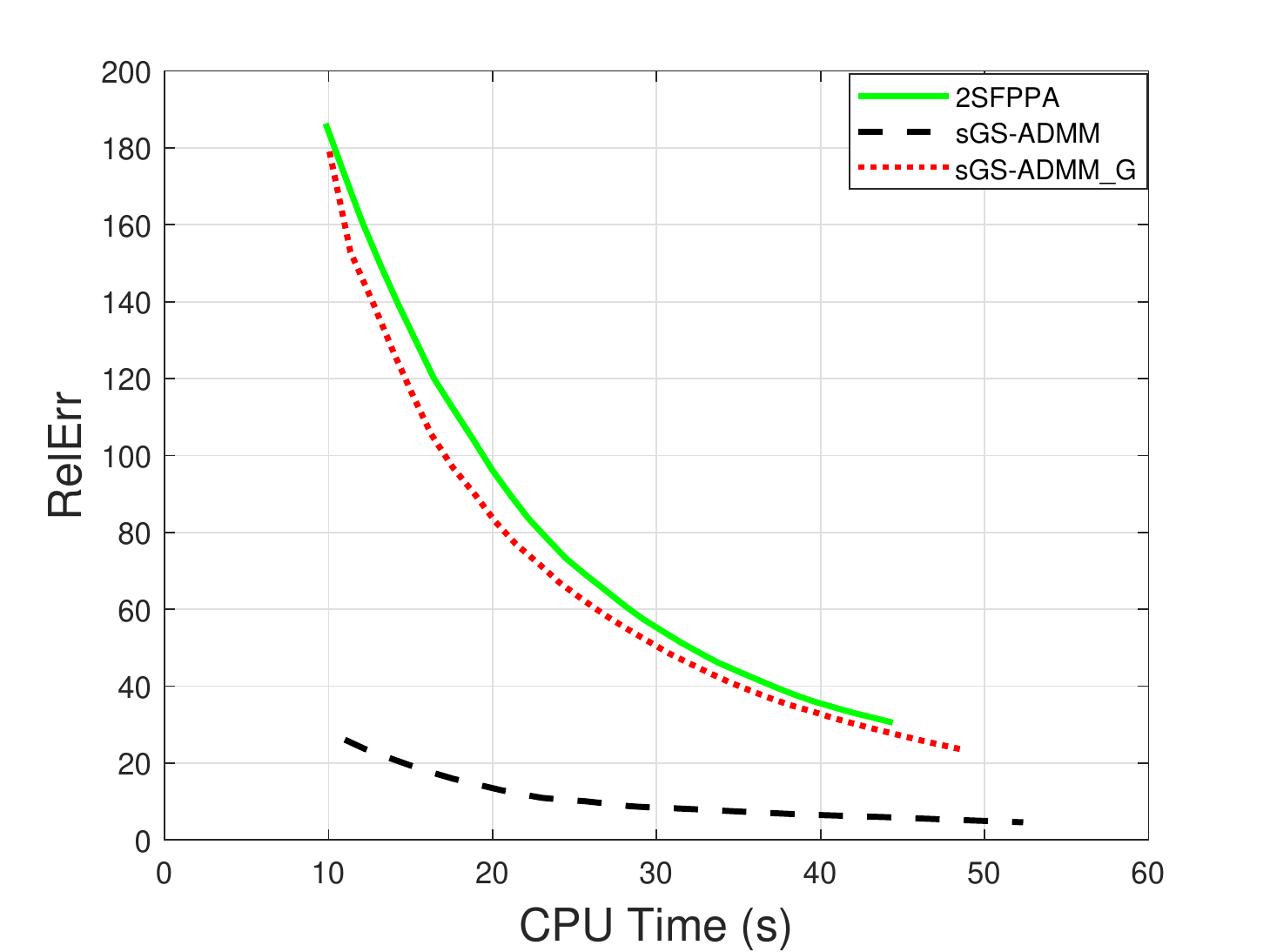}}\\
\subfigure[]{\includegraphics[width=0.3\textwidth,height=0.3\textwidth]{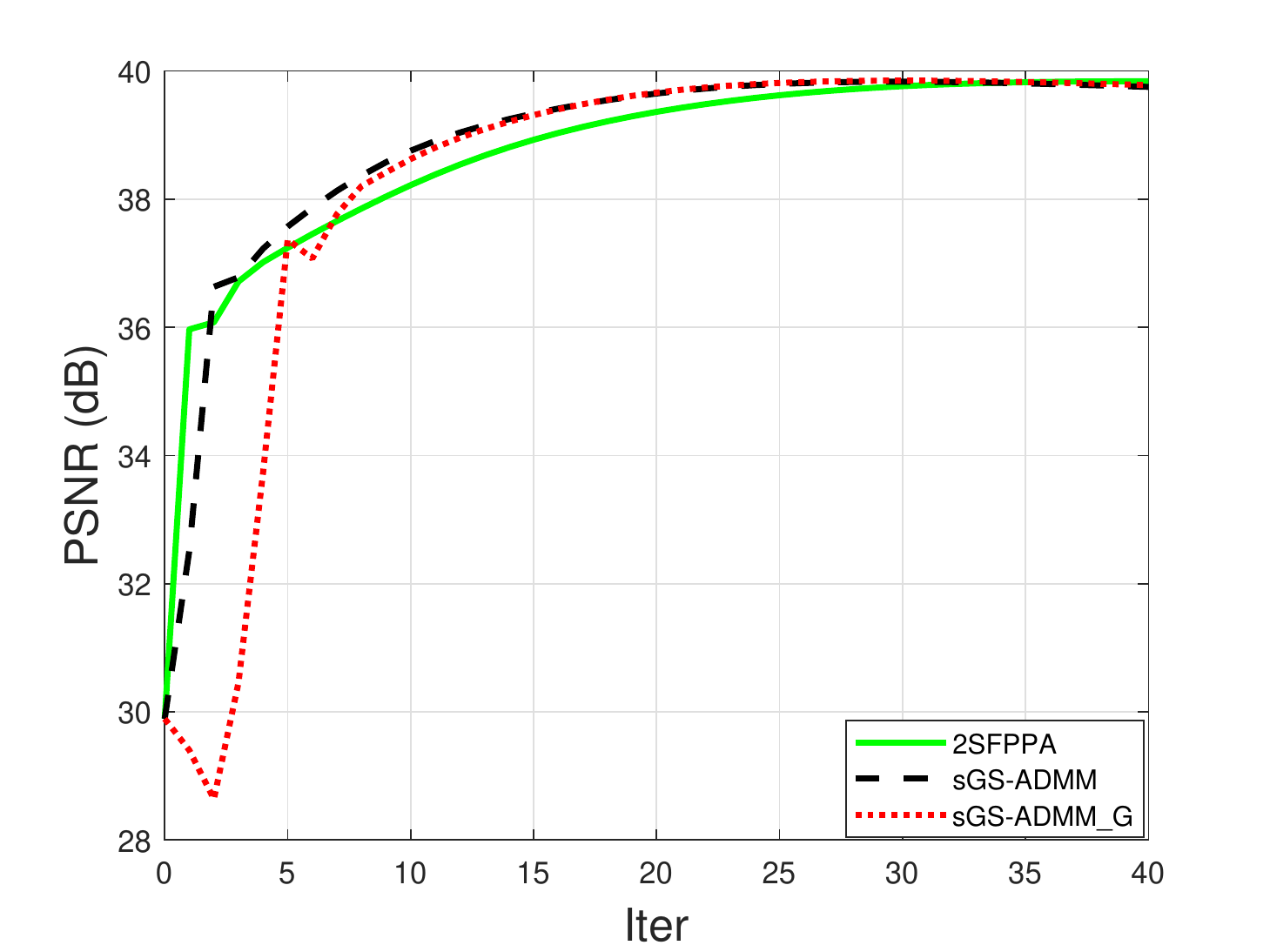}}
\subfigure[]{\includegraphics[width=0.3\textwidth,height=0.3\textwidth]{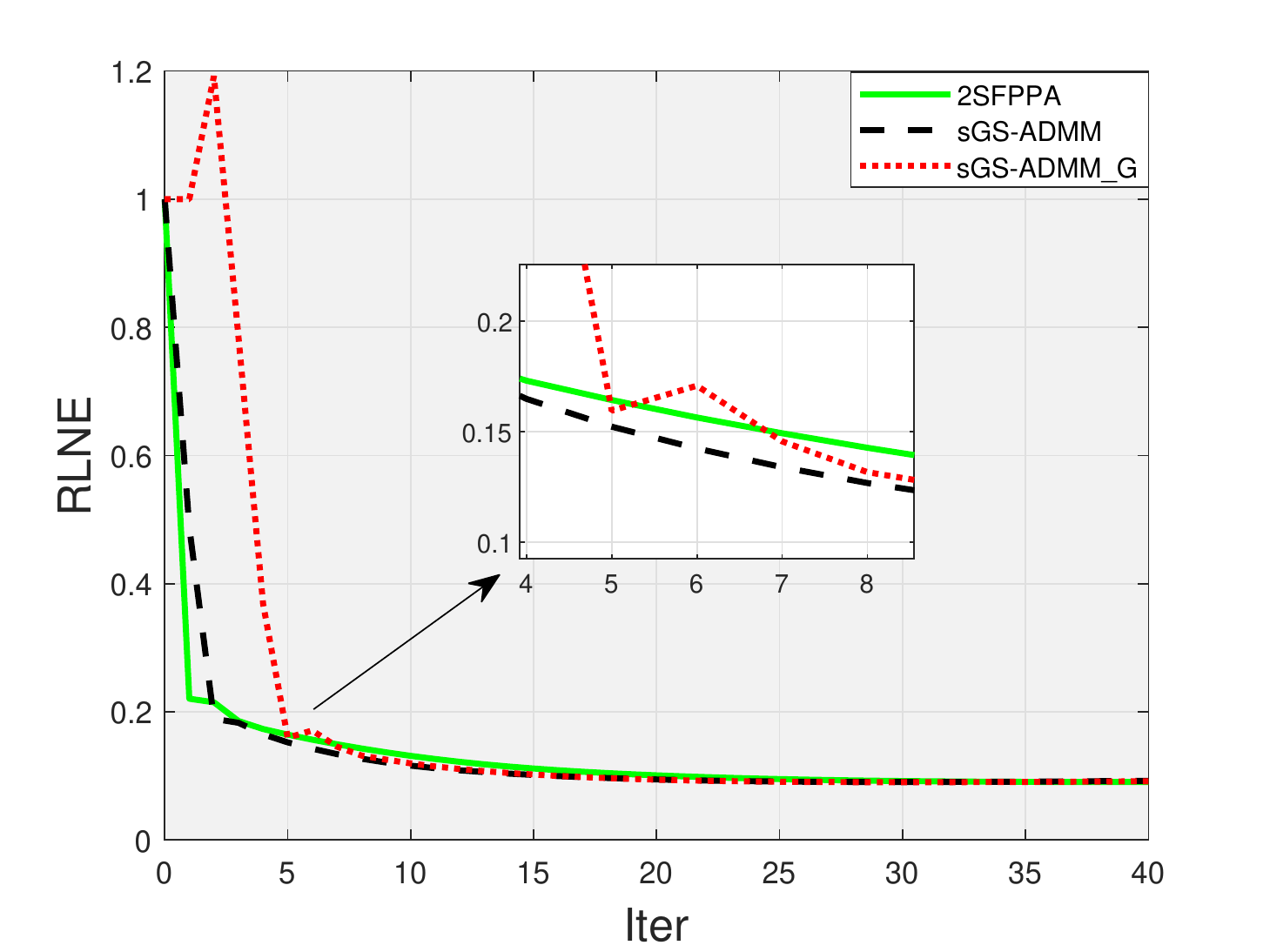}}
\subfigure[]{\includegraphics[width=0.3\textwidth,height=0.3\textwidth]{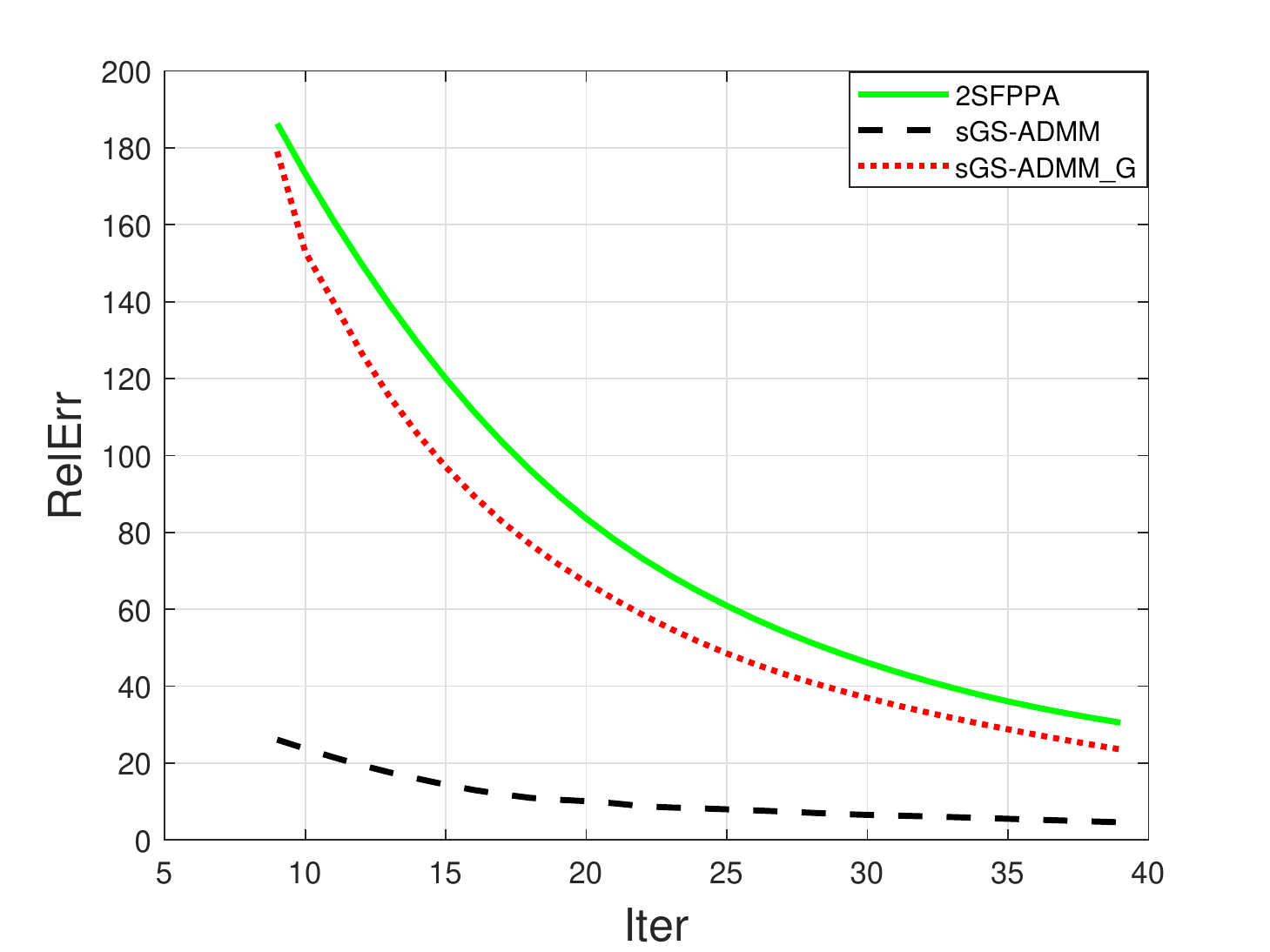}}
\end{tabular}
\end{center}
\caption{{\scriptsize Reconstruction results of the brian image with size $512 \times 512$. (a) PSNR versus computational time;
(b) RLNE versus computational time;
(c) RelErr versus computational time;
(d) PSNR versus number of iterations; (e) RLNE versus number of iterations;
(f) RelErr versus number of iterations.}}
\label{fig54}
\end{figure}

\subsection{Experiments using other undersampling patterns}

It is well known that the undersampling patterns are  very important to reduce reconstruction errors in MRI. In this part, we
aim to show that capabilities of sGS-ADMM and sGS-ADMM\_G are not limited to the pseudo radial sampling as tested previously,
which also suitable for other undersampling patterns.
In this subsection, we conduct experiments based on two typical patters named the Cartesian sampling by random phase
encoding of sampling rate $32.81\%$ and 2D random sampling  of sampling rate $30\%$ listed at the left and the right of Figure \ref{fig6}, respectively.
Besides, we test on six frequently used MR images in the literature with size $256\times 256$ as shown in Figure \ref{fig7}.
Moreover, the Gaussian noise as same as the aforementioned one is considered in these experiments. The numerical
results derived by each algorithm with regarding to the final RLNE, RelErr, and PSNR values within $40$ iterations are summarized in Table \ref{tab2}.
\begin{figure}[h]
\begin{center}
\begin{tabular}{c}
\subfigure[]{\includegraphics[width=0.3\textwidth,height=0.3\textwidth]{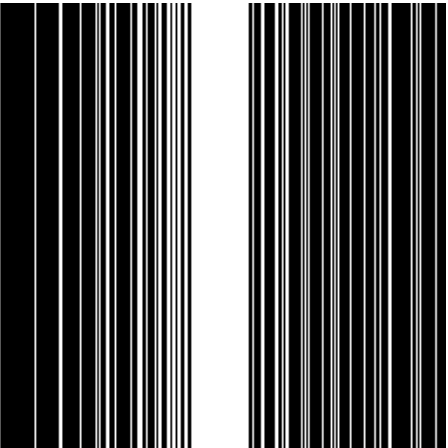}}
\subfigure[]{\includegraphics[width=0.3\textwidth,height=0.3\textwidth]{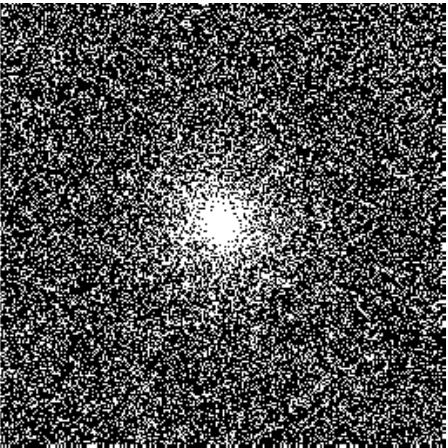}}
\end{tabular}
\end{center}
\caption{{\scriptsize Sampling patterns to be used. (a) Cartesian sampling pattern with $32.81\%$ data being sampled; (b) 2D random sampling pattern with $30\%$ data being sampled.}}
\label{fig6}
\end{figure}

\begin{figure}[h]
\begin{center}
\begin{tabular}{c}
\subfigure[]{\includegraphics[width=0.3\textwidth,height=0.3\textwidth]{phantom_1.pdf}}
\subfigure[]{\includegraphics[width=0.3\textwidth,height=0.3\textwidth]{brain1_1.pdf}}
\subfigure[]{\includegraphics[width=0.3\textwidth,height=0.3\textwidth]{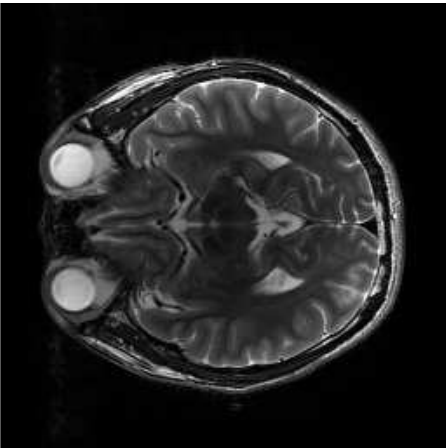}}\\
\subfigure[]{\includegraphics[width=0.3\textwidth,height=0.3\textwidth]{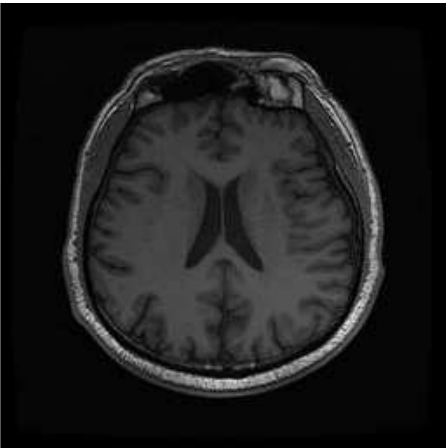}}
\subfigure[]{\includegraphics[width=0.3\textwidth,height=0.3\textwidth]{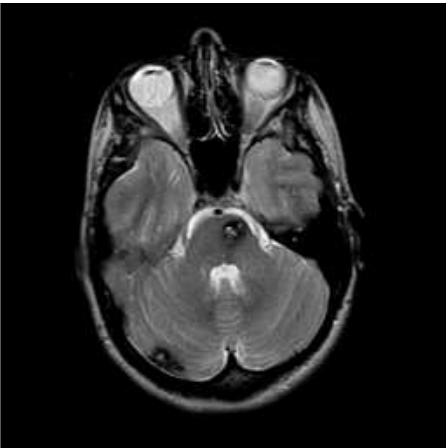}}
\subfigure[]{\includegraphics[width=0.3\textwidth,height=0.3\textwidth]{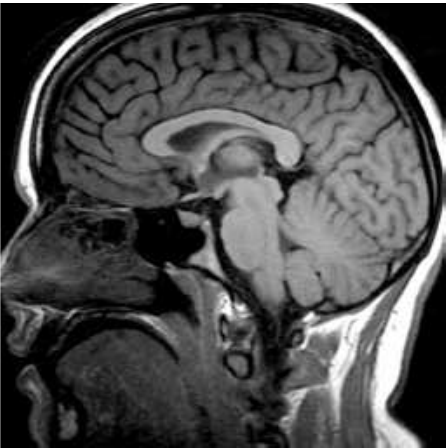}}
\end{tabular}
\end{center}
\caption{\scriptsize  MR images to be tested with two different sampling patterns illustrated in Figure \ref{fig6}.}
\label{fig7}
\end{figure}

From this table, we observe that the three algorithms can effectively reconstruct these $9$ kinds of MR images with different undersampling patterns. More precisely,
the RLNE and RelErr values derived by sGS-ADMM and sGS-ADMM\_G are always smaller than the one derived by 2SFPPA, which indicates that
sGS-ADMM and sGS-ADMM\_G always produce higher quality resolutions and in turn demonstrates that both presented algorithms are the winner in
reconstructing MR images.

\begin{table}[h]
\centering {\scriptsize\caption{Numerical results of all the algorithms under other two different samplings}
\begin{tabular}{cc | ccc | ccc | ccc }
\hline
\multicolumn{2}{c|}{} & \multicolumn{3}{c|}{2SFPPA} & \multicolumn{3}{c|}{sGS-ADMM} & \multicolumn{3}{c}{sGS-ADMM\_G}\\
\hline
Data&mask&RLNE&PSNR&RelErr&RLNE&PSNR&RelErr&RLNE&PSNR&RelErr\\
\hline
\multirow{2}{0.8cm}{Figure 9(a)}
&Cartesian &0.0442&46.8&23.7 &0.0465&46.1 &20.7 &0.0411&47.0&18.7\\
&2D random &0.0628&47.1&15.0&0.0677&46.6&13.7&0.0573&47.5&13.3\\
\hline
\multirow{2}{0.8cm}{Figure 9(b)}
&Cartesian&0.2044&42.0&3.02&0.2012&42.1&2.42&0.2034&42.1&1.77\\
&2D random&0.1685&45.2&3.06&0.1637&45.4&2.37&0.1662&45.3&1.82\\
\hline
\multirow{2}{0.8cm}{Figure 9(c)}
&Cartesian&0.1668&44.1&2.54&0.1644&44.2&2.42&0.1660&44.1&1.76\\
&2D random&0.1633&45.8&2.89&0.1606&45.9&2.51&0.1616&45.9&1.95\\
\hline
\multirow{2}{0.8cm}{Figure 9(d)}
&Cartesian&0.1166&47.0&5.50&0.1168&46.7&3.68&0.1166&47.0&3.13 \\
&2D random&0.1164&48.1&3.98&0.1179&48.1&3.71&0.1162&48.2&3.12\\
\hline
\multirow{2}{0.8cm}{Figure 9(e)}
&Cartesian&0.1220&44.1&4.45&0.1192&44.3&3.13&0.1208&44.2&2.68\\
&2D random&0.1068&47.1&3.72&0.1044&47.2&3.28&0.1047&47.2&2.71\\
\hline
\multirow{2}{0.8cm}{Figure 9(f)}
&Cartesian&0.1207&42.1&3.02&0.1180&42.3&2.66&0.1202&42.2&1.75\\
&2D random&0.1400&43.4&3.70&0.1364&43.5&2.06&0.1388&43.4&1.93\\
\hline
\end{tabular}\label{tab2}
}
\end{table}

\section{Conclusions}\label{finsec}

The extensive applications of compressed sensing MRI in clinical diagnosis has attracted much attention by many experts and scholars.
It is generally believed in this communities that an advanced reconstruction algorithm plays a crucial rule in decreasing the acquisition time.
In this paper, we took a dual approach to present a couple of efficient reconstruction algorithms for minimizing the sum of an $\ell_1$-norm of wavelet transformation term and TV regularization term. A series of numerical results on simulated phantom data and real brain imaging data with different undersampling patterns demonstrated the superior performance of sGS-ADMM and sGSADMM\_G over the state-of-the-art solver 2SFPPA. The successes of both algorithms mainly lied in the successful using of the novel sGS
technique which skillfully overcame the nonconvergent defect of traditional ADMM according to a very slightly computing cost.
With the attractive theoretical properties and the encouraging numerical performance, we believe that the sGS based
ADMMs should have more potential applications in the filed of CS in the near future.

\section*{Acknowledgements}
The work of Y. Xiao is supported by the National Natural Science Foundation of China (Grant No. 11971149).
The work of H. Zhang is supported by the National Natural Science Foundation of
China (Grant No. 11771003).

\section*{References}
\bibliography{mrisgs}

\end{document}